\documentclass[12pt]{article}

\usepackage{amsthm}
\usepackage{amsmath}
\usepackage{amssymb}
\usepackage{comment}
\usepackage{graphicx}
\usepackage{xr}
\usepackage{framed}
\usepackage{color}
 \usepackage[titletoc,toc,title]{appendix}

\usepackage{thmtools}
\usepackage{thm-restate}

\usepackage{hyperref}

\usepackage{cleveref}
\usepackage[T1]{fontenc}
\usepackage[utf8]{inputenc}
\usepackage{authblk}

\usepackage{enumitem}

\usepackage{subcaption}

\usepackage[linesnumbered,ruled,vlined]{algorithm2e}
\usepackage{lineno}
\usepackage[nottoc]{tocbibind}

\usepackage{csquotes}

\title{Extending Drawings of Graphs to Arrangements of Pseudolines}
\author[1]{Alan Arroyo\thanks{Supported by CONACYT}}
\author[2]{Julien Bensmail \thanks{ERC Advanced Grant GRACOL, project no. 320812}}

\author[1]{R. Bruce Richter\thanks{Supported by NSERC}}
\affil[1]{Department of Combinatorics and Optimization, University of Waterloo, Canada}
\affil[2]{Université Côte d’Azur, CNRS, Inria, I3S, France.
}

\theoremstyle{definition}
\theoremstyle{notation}

\newtheorem{theorem}{Theorem}[section]
\newtheorem{lemma}[theorem]{Lemma}
\newtheorem{corollary}[theorem]{Corollary}

\newtheorem{claim}{Claim}

\newtheorem*{usefulfact}{Useful Fact}
\newtheorem*{case}{Case}
\newtheorem{observation}[theorem]{Observation}

\newtheorem{question}{Question}
\newtheorem*{disentangling_step}{Disentangling Step}
\newtheorem*{face_escaping}{Face-Escaping Step}
\newtheorem*{exterior_meeting}{Exterior-Meeting Step}

\newcommand{\junesix}[1]{{\color{black}#1}}
\newcommand{\junefif}[1]{{\color{black}#1}}

\newcommand{\septtwotwo}[1]{{\color{black}#1}}
\newcommand{\octtwofour}[1]
{{\color{black}#1}}
\newcommand{\feba}[1]
{{\color{black}#1}} %%% february 9 2018
\newcommand{\marchtwo}[1]
{{\color{black}#1}}  %%% march 2
\newcommand{\rev}[1]
{{\color{black}#1}}  %%% april 9 2018
\newcommand{\ther}[1]
{{\color{black}#1}}  %%% april 9 2018

\newcommand{\dise}[1]
{{\color{black}#1}}  %%% April 8, 2018
\newcommand{\disen}[1]
{{\color{black}#1}}  %%% April 9, 2018
\newcommand{\disent}[1]
{{\color{black}#1}}  %%% April 9, 2018
\newcommand{\disenta}[1]
{{\color{black}#1}}  %%% April 9, 2018

\newcommand{\apriloneseven}[1]
{{\color{black}#1}}  %%% April 17, 2018

\newcommand{\apriloneeight}[1]
{{\color{black}#1}}  %%% April 18, 2018

\begin{document}
\maketitle

\begin{abstract}
A {\em pseudoline} is a homeomorphic image of the real line in the plane so that its complement is disconnected. An {\em arrangement of pseudolines} is a set of pseudolines in which every two cross exactly once. A drawing of a graph is {\em pseudolinear} if the edges can be extended to an arrangement of pseudolines.
In the recent study of crossing numbers, pseudolinear drawings have played an important role as they are a natural combinatorial extension of rectilinear drawings. 
A characterization of the pseudolinear drawings of $K_n$ was found recently. We 
extend this characterization to all graphs, by describing the set of minimal forbidden subdrawings for pseudolinear drawings. 
Our characterization also leads to a polynomial-time algorithm to recognize pseudolinear drawings and construct the pseudolines when it is possible. 
\end{abstract}

\section{Introduction}

A {\em pseudoline\/} is an unbounded open arc in the plane whose complement
is disconnected. In particular, lines are pseudolines, and any pseudoline
is the image of a line under a homeomorphism of the plane into itself. An
{\em arrangement of pseudolines} is a set of pseudolines in which every two intersect
in exactly one point, and their intersection point is a crossing.  A drawing of a graph $G$ is {\em pseudolinear\/} if there is an arrangement of pseudolines consisting of a different pseudoline for each edge and each edge is contained in its pseudoline.

In this work we characterize pseudolinearity of a drawing of any graph, not just $K_n$: the drawing must be good \apriloneeight{(defined below)} and not contain any of the configurations in Figure \ref{obstructions}.  
%A drawing is {\em stretchable} if it is homeomorphic to a {\em }
Thomassen \cite{T88} already observed that many of the drawings in Figure \ref{obstructions} %are not homeomeomorphic to rectilinear drawings, i.e. they 
are
obstructions \apriloneeight{for a drawing to be homeomorphic to a rectilinear drawing;} they are also obstructions for pseudolinearity. 

  \begin{figure}[ht]

    \centering
    
    \includegraphics[scale=0.3]{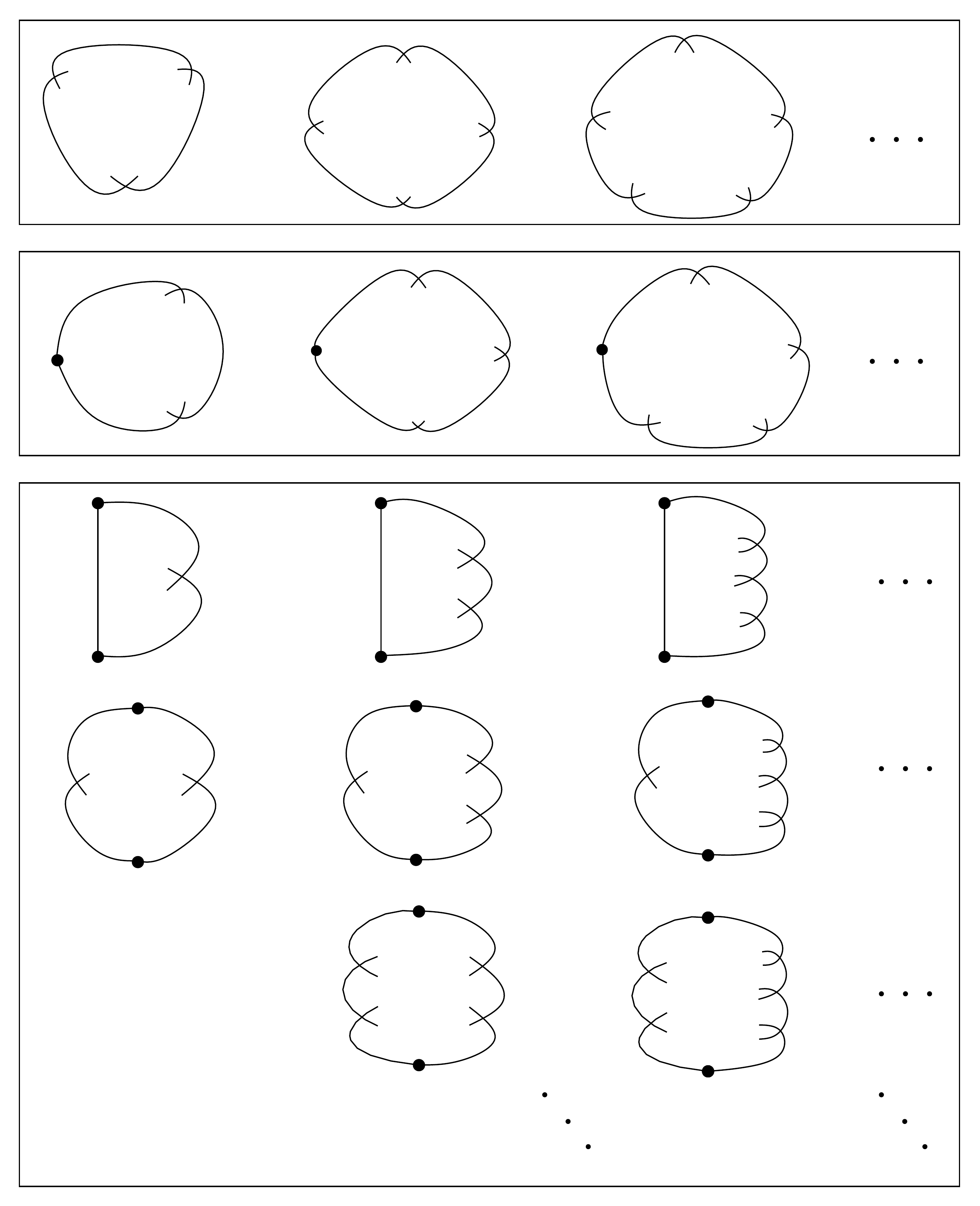}
    \caption{Obstructions to pseudolinear drawings.   }
 	\label{obstructions}
\end{figure}

We have been unable to find any literature that suggests even the possibility of a characterization of pseudolinearity. Moreover, informal conversations with colleagues seemed to be more along the lines of finding more obstructions.% Our main result, formally stated below, is that the drawings in Figure \ref{obstructions} are the remaining obstructions to pseudolinearity.

A rectilinear drawing of a graph is one in which edges are drawn using straight line segments, and more generally, a stretchable drawing is one that is homeomorphic to a rectilinear drawing. F\'ary's Theorem \cite{F48,S51,W36}, a classic result in graph theory, asserts that drawings of simple graphs with no crossings between edges are stretchable.

%A rectilinear drawing of a graph is one in which edges are drawn using straight line segments, and more generally, a stretchable drawing is one that is homeomorphic to a rectilinear drawing. F\'ary's Theorem \cite{F48,S51,W36}, a classic result in graph theory, asserts that drawings of simple graphs with no crossings between edges are stretchable.

In \cite{T88}, Thomassen extended F\'ary's Theorem  by characterizing stretchable drawings of graphs in which every edge is crossed at most once: 
In addition to being a {\em good} drawing (that is, no edge self-intersects and no two edges have two points -- either crossings or common endpoints -- in common), there are two forbidden configurations, shown in Figure \ref{B_and_W}.

\begin{figure}[ht]

    \centering
    
    \includegraphics[scale=0.4]{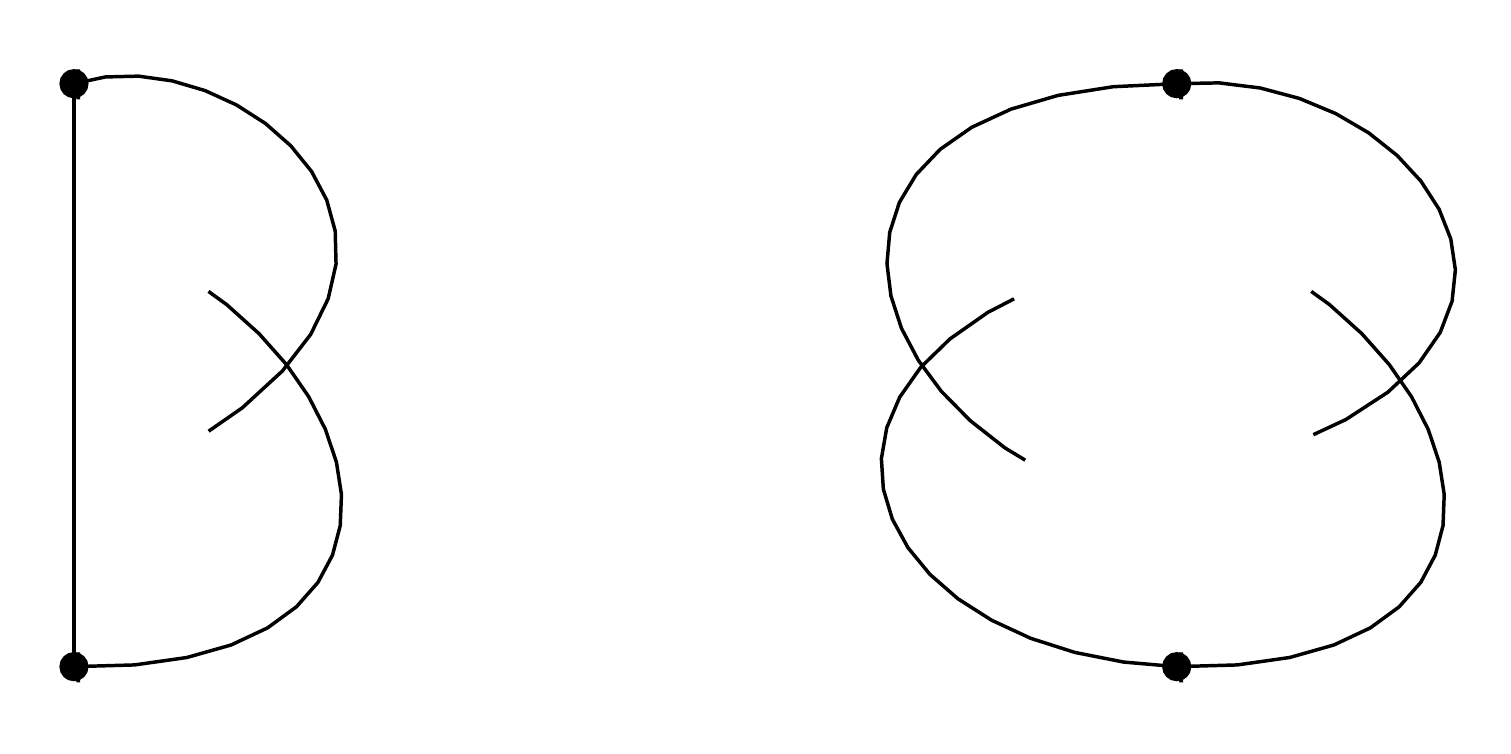}
    \caption{$B$ and $W$ configurations.}
	 	\label{B_and_W}
\end{figure}

Thomassen's characterization is a partial answer to the general problem
of determining which drawings are stretchable. There is not likely to be a
complete characterization, as Mn\"ev \cite{M85,M88} showed that the closely related problem of stretchability of arrangements of pseudolines is NP-hard (in fact $\exists\mathbb R$-hard). This easily implies that stretchability of graph drawings is NP-hard.

%A {\em pseudoline\/} is an unbounded open arc in the plane whose complement
%is disconnected. In particular, lines are pseudolines, and any pseudoline
%is the image of a line under a homeomorphism of the plane into itself. An
%{\em arrangement of pseudolines} is a set of pseudolines in which every two intersect
%in exactly one point, and their intersection point is a crossing. 

\apriloneseven{
The study of arrangements of pseudolines  was initiated by Levi and Ringel \cite{Levi1926,Ringel1956}, and propagated by  Gr\"unbaum's popular monograph {\em Arrangements and spreads} \cite{G72}.
}

Arrangements of pseudolines have played an important role in the study of the crossing number of $K_n$. A drawing of a graph $G$ is {\em pseudolinear\/} if there is an arrangement of pseudolines consisting of a different pseudoline for each edge and each edge is contained in its pseudoline. The original, independent proofs by \'Abrego and Fern\'andez-Merchant \cite{AF05} and by Lov\'asz et al \cite{LVWW04} that a rectilinear drawing of $K_n$ has at least 
\[
\frac 14\left\lfloor \frac{\mathstrut n}{\mathstrut 2}\right\rfloor 
\left\lfloor \frac{\mathstrut n-1}{\mathstrut 2}\right\rfloor
\left\lfloor \frac{\mathstrut n-2}{\mathstrut 2}\right\rfloor
\left\lfloor \frac{\mathstrut n-3}{\mathstrut 2}\right\rfloor
\]
crossings in fact applies to pseudolinear drawings of $K_n$. The substantial progress on computing the rectilinear crossing number of $K_n$ has continued this approach and has lead to further study of pseudolinear drawings \cite{BLPRS07,HLS17,AMRS15,AHPSV15,Cardinal2016}. \apriloneseven{Since the pseudolinear obstructions are also rectilinear obstructions, we wonder if this work might shed light on rectilinear drawings of graphs.  For example, Thomassen characterizes when a drawing of a graph in which each edge has at most one crossing is homeomorphic to a rectilinear drawing.  Our result shows that this is if and only if the drawing is pseudolinear.  (Pseudolinearity is obviously necessary; that it is sufficient is a little surprising.)} 

There have been recent, independent characterizations of pseudolinear drawings of $K_n$ \cite{AMRS15,AHPSV15}. The simpler of the equivalent descriptions is that the drawing is good and that it does not contain
%there are three forbidden configurations: the first two diagrams in Figure \ref{B_and_W} and 
the unique (up to homeomorphism) \apriloneeight{good} drawing of $K_4$ having the crossing incident with the infinite face. 

%It is the purpose of this work to provide a characterization of pseudolinearity of any graph, not just $K_n$. The first two drawings in Figure \ref{B_and_W} are still obstructions. Thomassen \cite{T88} already observed that many of the drawings in Figure \ref{obstructions} are obstructions for stretchability; they are also obstructions for pseuolinearity. 

%We have been unable to find any literature that suggests even the possibility of a characterization of pseudolinearity. Moreover, informal conversations with colleagues seemed to be more along the lines of finding more obstructions. Our main result, formally stated below, is that the drawings in Figure \ref{obstructions} are the remaining obstructions to pseudolinearity.

\apriloneeight{Our main} theorem is best presented in the context of strings in the plane. 
 A {\em string} $\sigma$ is the image $f([0,1])$ of  a continuous function  $f:[0,1]\rightarrow \mathbb{R}^2$ that restricted to \feba{$(0,1)$} is injective; in other words, strings are arcs that are  allowed  to self-intersect only at their {\em ends} $f(0)$ and $f(1)$. If no such self-intersection exists, then $\sigma$ is {\em simple}.  Most of the time we will  consider simple strings, although  considering  non-simple strings will come 
\junesix{
in}
handy for technical reasons.

%It is important to understand how two strings $\sigma$ and $\sigma'$ may intersect. Suppose that $x\in  \sigma\cap \sigma'$  is not an end of either $\sigma$ or $\sigma'$, and that $x$ is so that $\{x\}$ is a connected component of 
%$ \sigma\cap \sigma'$.   Consider  a small open subarc $\alpha$ of $\sigma$ centered at $x$. If $\alpha$  includes points on both sides of a small arc of $\sigma'$ centered at $x$. Otherwise $x$ is a tangential point.
A set of strings $\Sigma$ is in {\em general position} if, 
%(i)  all the strings in $\Sigma$ are simple; 
 for every two strings $\sigma$, $\sigma'\in \Sigma$ (i) $\sigma\cap\sigma'$ is a finite set of points in $\mathbb{R}^2$; and (ii) each point in $\sigma\cap \sigma'$ is either a crossing between $\sigma$ and $\sigma'$, or an end of either $\sigma$ or $\sigma'$. For instance, the set of edge-arcs of a good drawing of a graph is a set of strings in general position, but not all the sets of strings in general position come in this fashion: a string might  include  end points of other strings in its interior.

For a set $\Sigma$ of strings in general position, its {\em underlying plane graph} $G(\Sigma)$ is the plane graph obtained from $\Sigma$ by replacing  the  crossings between strings and the end points of every string in $\Sigma$ by vertices.  Our   main result
\junesix{
below}
characterizes when
\junesix{
a}
set of strings in general position  can be extended to an arrangement of pseudolines. 

\begin{theorem}\label{MAIN}
A set of strings $\Sigma$ in general position  can be extended to an arrangement of pseudolines if and only if, for each cycle $C$ in the underlying plane graph $G(\Sigma)$ of $\Sigma$, there are at least three vertices with the property that the edges incident to the vertex that are  included 
\junesix{in the closed}
disk bounded by $C$ belong to distinct strings in \junefif{$\Sigma$}. 
\end{theorem}

For instance, let  $C$ be the unique cycle in  the underlying plane  graph in any of the  drawings in Figure \ref{obstructions}. There are at most two vertices \ther{of $G$} in $C$,  represented as black dots. The strings incident with such a vertex are distinct and contained in the closed disk bounded by $C$. The vertices represented as crossings do not satisfy 
\junesix{this}
property: \septtwotwo{they are incident with four edges in the disk bounded by $C$, and}  these four edges consist of two strings that cross at this vertex.	
 Theorem \ref{MAIN} implies that none of the drawings in Figure \ref{obstructions} is pseudolinear. Surprisingly, we will show, as a consequence of Theorem \ref{MAIN}, that every non-pseudolinear drawing contains one of the configurations in Figure \ref{obstructions} as a subdrawing.

  \begin{restatable}{theorem}{minconf}
  \label{thm:min_conf}
  \junefif{Let $D$ be a non-pseudolinear good drawing of a graph $H$.}
  Then there is a subset $S$ of edge-arcs in $\{D[e]\;:\; e\in E(H)\}$, such that each $\sigma\in S$ has a substring $\sigma'\subseteq \sigma$ for which $\bigcup_{\sigma\in S} \sigma'$ is one of the drawings in Figure \ref{obstructions}.
\end{restatable}

Cycles that  have fewer than three vertices as in Theorem 1 are  the {\em obstructions} of $G(\Sigma)$  \ther{(this definition will be made more precise at the beginning of Section \ref{sec:main})}.
Showing that when $G(\Sigma)$ has obstructions, then $\Sigma$ cannot be extended to an arrangement of pseudolines,  is the first part of Section \ref{sec:main}. The rest of Section \ref{sec:main} is devoted to show that if $G(\Sigma)$ has no obstructions, then $\Sigma$ can be extended to an arrangement of pseudolines. The proof of two technical lemmas used in the proof of Theorem \ref{MAIN} are deferred to Section \ref{sec:teclemmas}. In Section \ref{sec:find_obs}, we describe a simple algorithm that finds an obstruction in polynomial time. In Section \ref{sec:Kn}, by applying Theorem \ref{MAIN}, we prove that a drawing
\junesix{of a}
complete graph $K_n$ is pseudolinear if and only if it does not contain the $B$ configuration in Figure \ref{B_and_W}. This result is equivalent to the characterizations of pseudolinear drawings of $K_n$ given  in \cite{AHPSV15} and \cite{AMRS15}, but its proof is  simpler. At the end, in Section \ref{sec:conclusions}, we show how Theorem \ref{thm:min_conf} easily follows from \marchtwo{Theorem} \ref{MAIN}, together with some concluding remarks.

\section{Proof of Theorem \ref{MAIN}}
\label{sec:main}

In this section, we use Lemmas \ref{2SIDES} and \ref{2SIDESEDGE} (proved in the next section)
to prove Theorem \ref{MAIN}.
As we enter  into the subject, we need some  notation that is useful in identifying an obstruction.
Let $C$ be a cycle of  a plane graph $G$ and let $v$ be a vertex of $C$. 
The  {\em  rotation at $v$ inside $C$}  is the counterclockwise ordered list $e_0,e_1...,e_k$ of edges incident with $v$ that are included in the closed disk bounded by $C$, with $e_0$ and $e_k$ both in $C$.  Likewise, the {\em rotation at $v$ outside $C$} is defined as  the counterclockwise ordered list $e_k,$ \septtwotwo{$e_{k+1},\dots,e_0$} of edges incident with $v$ included in the closure of the exterior of $C$.
%We remark that the rotations of a vertex inside and outside a cycle contain the two edges in the cycle that are  incident to the vertex. 

%\begin{comment}
\begin{figure}[h]
	\centering
	\begin{subfigure}[b]{0.45\textwidth}
	\centering
		\includegraphics[scale=0.5]{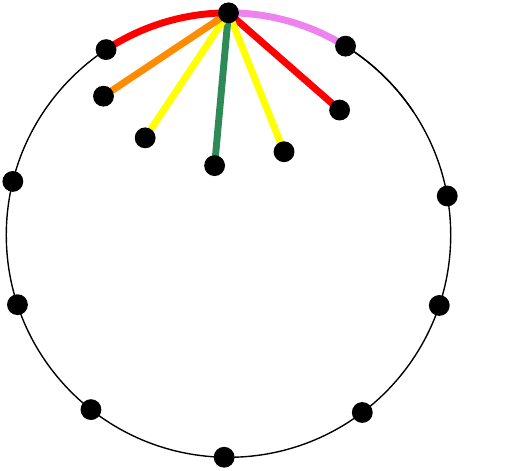}
		\caption{A reflecting vertex}
		\label{subfig:reflecting}
		\end{subfigure}
		~
	\begin{subfigure}[b]{0.45\textwidth}
		\centering
		\includegraphics[scale=0.5]{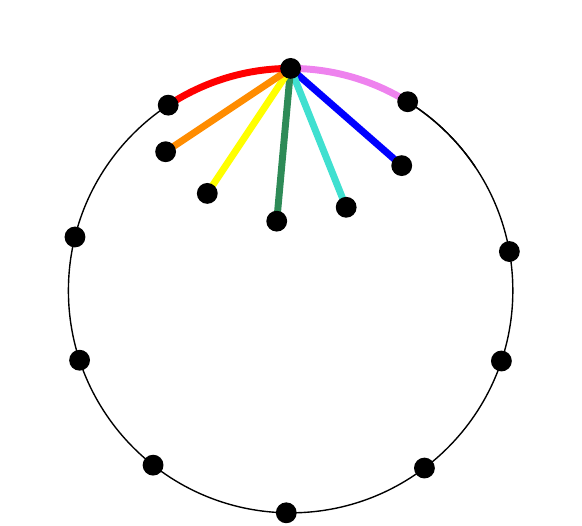}
		\caption{A rainbow}
		\label{subfig:rainbow}
		\end{subfigure}
		
		\caption{A representation of reflecting and rainbow vertices, where  each  string in $G(\Sigma)$  has assigned  a unique colour.}
\end{figure}
%\end{comment}

\junesix{
In the case $G=G(\Sigma)$ for some set $\Sigma$  of
strings in general position, a vertex $v$ in a cycle $C$ of $G(\Sigma)$ is {\em reflecting} in $C$ if at least two
edges in the rotation \septtwotwo{at} $v$ inside $C$ belong to the same string \ther{(Figure \ref{subfig:reflecting})}. The alternative is that $v$ is a {\em rainbow}, in which case all
the edges of its rotation inside $C$ are in different strings \ther{(Figure \ref{subfig:rainbow})}. In these terms, an {\em obstruction}
is a cycle with at most two rainbows.
}

\subsection{Sets of strings with obstructions are not extendible}
\label{subsec:obstructions_not_extendible}

The following observation will be used in this subsection and also in Theorem \ref{thm:Kn}. If $C$ is a cycle in $G(\Sigma)$, where
 $\Sigma$ is a set of strings in general position, then  $\delta(C)$ is the set of vertices in $C$ for which their two incident edges in $C$ belong to two distinct strings in $\Sigma$. Note that if $|\delta(C)|<3$ for some cycle $C$, then either \septtwotwo{$|\delta(C)|=2$} and two strings intersect more than once, or \septtwotwo{$|\delta(C)|\leq1$} and some string is self-crossed. \junefif{ Both these possibilities are forbidden in good drawings. }

\begin{observation}\label{obs:min_delta}

Let $\Sigma$ be a set of simple strings in general position in which every two strings intersect at most once. \feba{Let:
\begin{itemize}
\item[(a)] $C$ be an obstruction of $G(\Sigma)$ for which $|\delta(C)|$ is as small as possible;
\item[(b)] $x\in\delta(C)$;
\item[(c)] $e$ be an edge in $C$  incident to $x$;
\item[(d)] $\sigma\in \Sigma$ be the string  containing $e$; and
\item[(e)] $\sigma'$ be the component of $\sigma\setminus e$ containing $x$.
\end{itemize}
}
 Then $\sigma'\cap C=\{x\}$.
\end{observation}
\begin{proof}
By  way of contradiction, suppose that $\sigma'\cap C$ includes a point distinct from $x$. This in particular  implies that  $\sigma'\neq \{x\}$, and, because $x\in \delta(C)$, the points  of $\sigma'\setminus\{x\}$ near $x$ are not in $C$. Let $P$ be the path in $G(\Sigma)$ obtained by traversing $\sigma'$, starting at $x$, and stopping the first time we encounter a point $y\in C\cap(\sigma'\setminus\{x\})$. Note that $y\in V(C)$ and that $P$ is drawn in either  the interior or the exterior of $C$.

 First, suppose that $P$ is drawn in the interior of $C$. Let $C_1$ and $C_2$ be the cycles obtained from the union of $P$ and one of the two $xy$-subpaths in $C$. We may assume $C_1$ includes $e$. 
 \junefif{Each of $C_1-P$ and $C_2-P$ has a vertex  in $\delta(C)$; otherwise one of $C_1$ or $C_2$ would be  included in at most two strings, implying that a string is self-crossing or two strings intersect twice. Therefore }
 $|\delta(C_1)|$ and $|\delta(C_2)|$ are strictly smaller than $|\delta(C)|$.
 \junefif{Then, by assumption,}
 $C_1$ and $C_2$  are not obstructions.
 
 None of the vertices in $P-y$ is a rainbow for $C_1$ (\feba{$P\subseteq \sigma'$ and $x$ is reflecting in $C_1$}, so \feba{ the interior rotations of the vertices in $P-y$} include two edges in $\sigma$). Since all the vertices in $C_1-V(P)$ that are rainbow in $C_1$ are also rainbow in $C$, $C_1$ has at most two rainbows in  $V(C_1)\setminus V(P)$. These last two observations and the fact that $C_1$ is not an obstruction, together imply that $C_1$ has three 
 rainbows\junefif{:}
 two of them
 \junefif{are}
 in $V(C_1)\setminus V(P)$ and the  other is $y$. 

Now we look at the rainbows in $C_2$. Because $C$ has two rainbows in $C_1-V(P)$ and any rainbow  in $V(C_2)\setminus V(P)$ for $C_2$ is rainbow for $C$,   $C_2$ has no rainbow   in $V(C_2)\setminus V(P)$. All the interior vertices of $P$ are reflecting in $C_2$, so $C_2$ has at most two rainbows. This contradicts that $C_2$ is not an obstruction.

Secondly, suppose that $P$ is drawn in the exterior of $C$. Let $C_{out}$ be the cycle bounding the outer face of $C\cup P$. The cycle $C_{out}$ is the union of $P$ and one of the two $xy$-paths in $C$, and, in both cases, \feba{as $x\in \delta(C)\setminus \delta(C_{out})$ and $P-y\subset \sigma$,} $|\delta(C_{out})|<|\delta(C)|$ . Every vertex in $P-y$ is reflecting in $C_{out}$ \ther{(this statement follows from the fact that the rotation of a vertex inside a cycle also includes the edges
of the cycle incident with the vertex)}. Moreover, every vertex in $V(C_{out})\setminus(V(P-y))$ that is a rainbow in $C_{out}$  is also a rainbow in $C$. These two facts  imply that $C_{out}$ has at most as many rainbows as $C$; hence $C_{out}$ is an obstruction. This contradicts	the fact that $C$ minimizes $|\delta|$. 
\end{proof}

Next we show that, if a set of strings contains an obstruction, then it is not pseudolinear.

\begin{observation}\label{obstruction_implies_nopseudolinear}

If $\Sigma$ is a set of strings in general position and  $G(\Sigma)$  has an obstruction, then $\Sigma$ cannot be extended to an arrangement of pseudolines.
\end{observation}
\begin{proof}
By  way of contradiction, suppose that there is a set of strings $\Sigma$ that can be extended to an arrangement of pseudolines and  $G(\Sigma)$ has an obstruction $C$.  Consider an extension of $\Sigma$ to an arrangement of pseudolines, and then  cut off the two infinite ends of each pseudoline to obtain
\junefif{a}
set of strings $\Sigma'$ extending $\Sigma$, and in which every two strings in $\Sigma'$ cross. In $G(\Sigma')$, there is a cycle $C'$ that represents the same simple closed curve as $C$. Because $C'$ is obtained from subdiving \septtwotwo{some edges} of $C$, $C'$ has fewer than three rainbows. Therefore, we may assume that $\Sigma=\Sigma'$ and $C=C'$. Now, the ends of every string in $\Sigma$ are degree-one vertices  in  the outer face of $G(\Sigma)$. 

As every string in $\Sigma$ is simple, and no two strings intersect more than once, $|\delta(C)|\geq 3$. \septtwotwo{We will assume that $C$ is chosen to minimize $|\delta(C)|$.}

 Since $C$ is an obstruction, there is at least one vertex  $x\in \delta(C)$ reflecting inside $C$. 
Let $e\in E(C)$ be an edge incident to $x$, and suppose that $\sigma$ is the string including $e$. 
Traversing $\sigma$ along $e$ through $x$, we encounter another edge $e'\subseteq \sigma$ incident to $x$.  Because $x\in \delta(C)$, $e'$ is not in $C$. Suppose that $e'$ is drawn in the outer face of $C$.
As $x$ is reflecting inside $C$, 
\junefif{there}
exists a string $\bar{\sigma}$ that includes two edges in the rotation at $x$ inside $C$. However, $\sigma$ and $\bar{\sigma}$ tangentially intersect at $x$, contradicting that the strings in $\Sigma$ are in general position. Therefore $e'$ is drawn inside $C$. 

Let $y$ be the end of $\sigma$ contained in the component of $\sigma\setminus e$ containing $x$. 
\feba{Since $|\delta(C)|$ is minimum, 
 Observation \ref{obs:min_delta} implies that
the component of $\sigma\setminus e$ having $x$ and $y$ as ends have all its points, with the exception of $x$,  in the inner face of $C$. However, $y$ is  drawn in the inner face of $C$, contradicting that  the ends of all the strings in $\Sigma$ are incident with the outer face of $G(\Sigma)$. }
\end{proof}

 \subsection{Extending sets of strings  with no obstructions}

In this subsection we  prove that  a set of strings with no obstructions can be extended to an arrangement of pseudolines. 
We restate Theorem \ref{MAIN} using  our new terminology.

\begin{theorem}
A set of strings $\Sigma$ in general position can be extended to an arrangement of pseudolines if and only if $G(\Sigma)$  has 
\junefif{no}
obstructions.
\end{theorem} 
\begin{proof}
  We  showed in Observation \ref{obstruction_implies_nopseudolinear} that if $G(\Sigma)$ has an obstruction, then $\Sigma$ cannot be extended to an arrangement of pseudolines. 
 \junefif{For the converse, suppose}
 that $G(\Sigma)$ has no obstructions. 

We start by reducing the proof to the case in which the point set $\bigcup \Sigma$ is connected. If $\bigcup \Sigma$ is not connected, then we add a simple string to $\Sigma$, connecting two points in distinct  components of $G(\Sigma)$, and so that it is included inside a face of $G(\Sigma)$. This operation:  reduces the number of components;  does not create obstructions; and 
ensures that 
 any  pseudolinear extension of the new set of strings shows the existence of one for $\Sigma$. We continue adding strings 
\junefif{in}
this way until we obtain a connected set of strings and we redefine $\Sigma$ to be this set. 
\junefif{Thus,}
we may assume $\bigcup \Sigma$ is connected.

Our proof is algorithmic, and consists of repeatedly applying  one of the three steps described below. 

\begin{itemize}
\item {\bf Disentangling Step.} If a string $\sigma\in \Sigma$ has an end $a$ with degree at least 2 in $G(\Sigma)$, then we slightly extend the $a$-end of $\sigma$ into one of the faces incident with $a$. 

\item {\bf Face-Escaping Step.} If a string $\sigma\in \Sigma$ has an end $a$ with degree 1 in $G(\Sigma)$, and  is incident with an inner face, then we extend the $a$-end of $\sigma$ until we intersect some point in the boundary of this face. 

\item {\bf Exterior-Meeting Step.} Assuming that all the strings in $\Sigma$ have their two ends in the outer face and these ends have degree 1 in $G(\Sigma)$, we extend the ends of two disjoint strings so that they meet in the outer face.

\end{itemize}

We can always perform at least one of these steps, unless the strings are pairwise intersecting and all of them have their ends in the outer face (in this case we extend their ends to infinity to obtain the desired  arrangement of pseudolines).
Each step increases the number of pairwise intersecting strings. 
Henceforth, our aim is to show that, as long as there is a pair of non-intersecting strings, then one of these three steps may be performed \ther{without adding an obstruction}. 
The proof is now  divided into three parts that can be 
 \junefif{read}
 independently.
\begin{disentangling_step}
 Suppose that $\sigma\in \Sigma$  has an end $a$  with degree at least $2$ in $G(\Sigma)$. Then we can
 extend the $a$-end of $\sigma$ into one of the faces incident to $a$ without creating an obstruction.
\end{disentangling_step}

\begin{proof}
\dise{
An edge $f$  of $G(\Sigma)$  incident with $a$ is a {\em twin} if there exists another edge $f'\neq f$ incident with $a$ such that both $f$ and $f'$ are part of the same string in $\Sigma$. Observe  that the edge $e_0\subseteq \sigma$ incident with $a$ is not a twin.
}

\dise{
The fact no pair of strings  tangentially intersect at $a$ tells us that if $(f_1,f_1')$ and $(f_2,f_2')$ are pairs of corresponding twins, then $f_1$, $f_2$, $f_1'$, $f_2'$ occur in this cyclic order for either the clockwise or counterclockwise rotation at $a$. Thus, we may assume that  the twins at $a$ are labeled as $f_1,\ldots, f_t, f_1',\ldots, f_t'$, and that this is their counterclockwise order occurrence when we follow the rotation at $a$ starting at $e_0$. In such a case,   $(f_i,f_i')$ is a pair of corresponding twins for $i=1,\ldots, t$. 
}

\dise{
In order to avoid tangential intersections when twins \disen{are} present,  every valid extension of $\sigma$ at $a$ must cross into the angle between $f_t$ and $f_1'$ not containing $e_0$. 
}

\dise{
Let $(e_1, \ldots, e_k)$ be the list of non-twin edges between $f_t$ and $f_1'$ in the counterclockwise rotation at $a$; this list might be empty. \disen{In the case there 
are no twins, we set $f_t$ 
and $f’_1$ both equal to 
$e_0$, so $(e_1,…,e_k)$ 
is all the edges incident 
with $a$ other than 
$e_0$.}
}

 We consider all the \dise{feasible} extensions for $\sigma$: for each \dise{$i\in\{1,...,k-1\}$}, we let $\Sigma_i$  be the set of strings obtained 
from extending  $\sigma$ by adding a small bit of arc $\alpha_i$ starting at $a$, and continuing into the face between $e_i$ and $e_{i+1}$. \disent{Let  $\Sigma_0$ be the set of strings obtained by adding an arc $\alpha_0$ in the face between $f_t$ and $e_1$, and let $\Sigma_k$ be obtained by adding an arc $\alpha_k$ in the face between $e_k$ and $f_1'$.
}

Seeking a contradiction, suppose that, for each $i\in \{0,...,k\}$,  $G(\Sigma_i)$ contains an obstruction $C_i$. The cycle $C_i$ does not include the bit of arc $\alpha_i$ as an edge, so $C_i$  is  a cycle in $G(\Sigma)$.
 This  cycle  is not an obstruction in $G(\Sigma)$, although it becomes one when we add $\alpha_i$. The reason explaining this conversion  is simple: in $G(\Sigma)$,  $C_{i}$  has  exactly three vertices not reflecting, and one of them is $a$.  After $\alpha_i$ is added, $a$ is  now reflecting in $C_i$ (witnessed by $\sigma$). 
 
Understanding how cycles with exactly three rainbows may behave in an obstruction-less set of strings is a crucial piece of the proof. In general, if $v$ is a vertex in  the underlying plane graph of a set of strings in general position, then a {\em near-obstruction at $v$} is a cycle with exactly three rainbows, and one of them is $v$.  Each of the cycles  $C_0$, $C_1$,...,$C_k$ above is a near-obstruction at $a$ in $G(\Sigma)$.
 \begin{figure}
 \centering
 \includegraphics[scale=0.4]{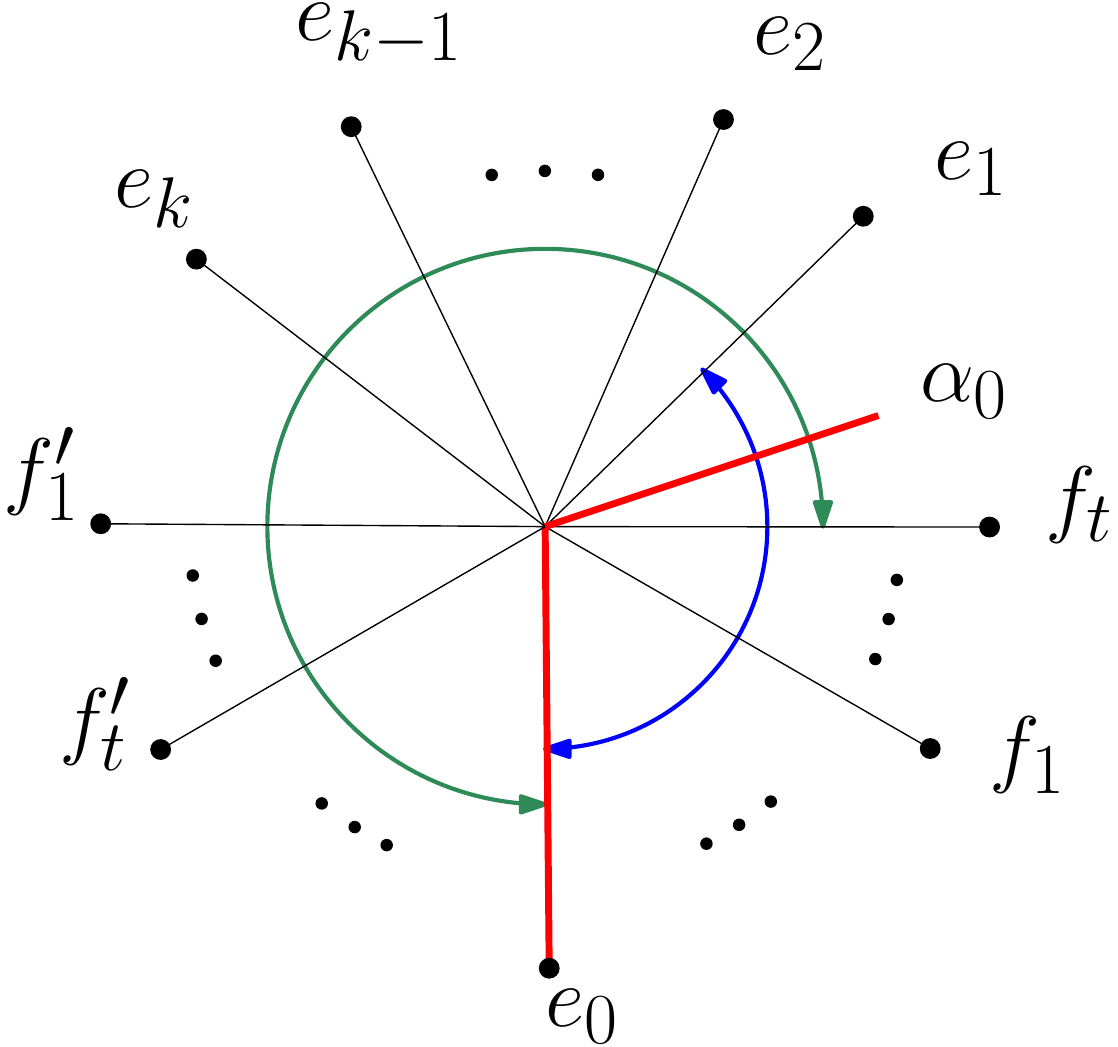}
  \caption{Substrings included in the disk bounded by $C_0$.}
  \label{fig:substrings_rotation}
 \end{figure}

 \disen{
 Both $e_0$ and $\alpha_0$ are on the disk bounded by $C_0$, and since $\alpha_0$ is not part of $C_0$, either $e_0, f_1, f_2,\ldots, f_t, e_1$ are on the same side  of $C_0$ (blue bidirectional arrow in Figure \ref{fig:substrings_rotation}) or all of $f_t, e_1, \ldots, e_k, f_1', f_2',\ldots, f_1',e_0$  are \disent{in} the same side of $C_0$ (green bidirectional arrow in Figure \ref{fig:substrings_rotation}).
 \disenta{
 Because $\alpha_0$ is the only edge between $f_t$ and $e_1$, we see that $e_1$ belongs to the first sublist (the blue one) and $f_t$ belongs to the second list (the green one).
 }
 %\disent{It is important to mention that $e_1$ belongs to the first sublist (the blue one) and that $f_t$ belongs to the second list (the green one) because $\alpha_0$ is the only edge between $f_t$ and $e_1$. }
 In the second case both $f_t$ and $f'_t$ are \disent{in} the disk bounded by  $C_0$, showing that $a$ is not a rainbow \disent{for $C_0$ in $\Sigma$}. Therefore, all of $e_0, f_1, f_2,\ldots, f_t, e_1$ are \disent{in} the disk bounded by $C_0$.
 
 }
 
 \dise{
  Regardless \disen{of} the presence or absence of twins, we know that $(e_0, e_1)$ occurs as a sublist of the rotation of $a$ inside $C_0$. A symmetric argument shows that $(e_k,e_0)$ occurs as a sublist of the rotation of $a$ inside $C_k$. 
 }
 
 \disen{
 
Since $(e_0,e_1)$ is a substring of the rotation at $a$ inside $C_0$ and $(e_0,e_1,\ldots,e_k,$ $f'_1)$ is not inside $C_k$, there is a largest $i\in \{0,1,\ldots, k-1\}$ such that $(e_0,\ldots, e_{i+1})$ is inside $C_i$. The choice of $i$ implies $(e_{i+1},\ldots, e_k,e_0)$ is inside $C_{i+1}$.
 
 }

\disen{
The next lemma states that the existencce of such a pair of cycles $C_i$ and $C_{i+1}$ is impossible, completing the proof. 
}
\end{proof}

  \begin{restatable}{lemma}{twosides}
\label{2SIDES}
Let  $\Sigma$ be a set of strings  in general position. Suppose that $C_1$ and $C_2$ are cycles in $G(\Sigma)$ that are near-obstructions  at $v$, so that  the rotation at $v$ inside $C_1$  includes (as a sublist) the rotation at $v$ outside $C_2$, and that the rotation at $v$ inside $C_2$ includes the rotation at $v$ outside $C_1$. Then $G(\Sigma)$ has an obstruction.
\end{restatable}

 We defer the proof of  Lemma \ref{2SIDES} to Section \ref{sec:teclemmas} as it  is technical and it deviates our attention from the proof of Theorem \ref{MAIN}.

%%%%%%%%%%%%%%%
%%%%%%%%%%%%%%

\begin{face_escaping}
 Suppose that there is a string $\sigma$ that has an end $a$ with  degree 1 in $G(\Sigma)$, and $a$  is  incident to an inner face $F$. 
 \junefif{ Then there is an extension $\sigma'$ of $\sigma$ from its $a$-end to a point in the boundary of $F$ such that the set $(\Sigma\setminus\{\sigma\})\cup \{\sigma'\}$ has no obstruction.}
\end{face_escaping}
\begin{proof}
\begin{figure}[ht]
 	
    \centering
    
    \includegraphics[scale=0.8]{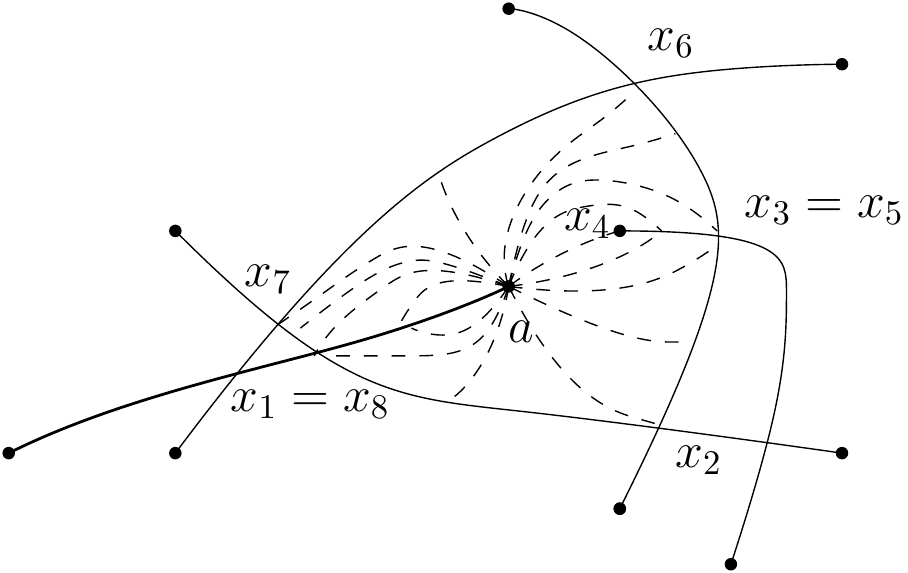}
    \caption{All possible extensions in the Face-Escaping Step.}
 \label{f_s_possible_extensions}
\end{figure}

\junefif{Let $W$ be the closed boundary walk $(x_0,e_1,\ldots,e_n,x_n)$  of $F$ such that $x_0=x_n=a$ and $F$ is to the left as we traverse $W$.}
\apriloneseven{
Let $P$  denote the list of points $(m_1, x_1,m_2,x_2,\ldots, x_{n-1}, m_n)$. For each point $p$ in $P$, let $\Sigma_p$ be the set of strings obtained from $\Sigma$ by extending the $a$-end of $\sigma$ adding an arc $\alpha_p$ connecting $a$ to $p$ in $F$ (see Figure \ref{f_s_possible_extensions}). 
}

\apriloneseven{
Figure \ref{f_s_middle} shows the importance of considering extensions meeting points in the middle an edge in the boundary of $F$,   as sometimes this is the only way for extending $\sigma$ without creating an obstruction. 
}

\begin{figure}[ht]
 	
    \centering
    
    \includegraphics[scale=0.6]{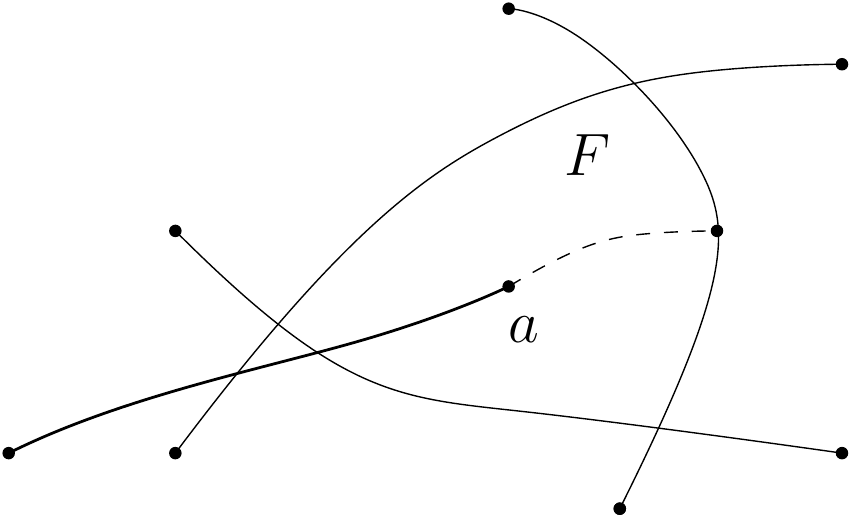}
    \caption{Face-Escaping Step.}
 \label{f_s_middle}
\end{figure}

Let $f_p$ be the edge $e_1\cup \alpha_p$ in $G(\Sigma_p)$; it has ends $x_1$ and $p$. Also, let $\sigma^p=\sigma\cup \alpha_p$. The existence of obstructions in $G(\Sigma_p)$  is independent of how we draw $\alpha_p$ inside $F$. We will take advantage of this fact later on in the proof.

%Sometimes   it is only possible  to extend $\sigma$  by using \feba{an}
%arc connecting $a$ to  a  vertex in the middle of an edge. Figure \ref{f_s_middle}  shows an example of this situation. \septtwotwo{ Joining $a$ to any $x_i$ produces a second crossing between two strings; the simple closed curves enclosed by the union of these two strings is an obstruction.}
% So for each edge  $e_i$,  let $m_i$  be a point in  $e_i$, between $x_{i-1}$ and $x_{i}$.  

%\junefif{Let $P$  denote the list of points $(m_1, x_1,m_2,x_2,\ldots, x_{n-1}, m_n)$. For each point $p$ in $P$, let $\alpha_p$ denote an arc in $F$ connecting $a$ to $p$. (A given element of the point set of $G(\Sigma)$ might occur more than once in $P$. In particular, it is always the case that $x_1=x_{n-1}$; see Figure \ref{f_s_possible_extensions}.)}

%For each point $p$ in $P$, we consider the \junefif{set of strings} $\Sigma_p$ obtained from $\Sigma$ by extending $\sigma$ using $\alpha_p$. 
%\junefif{Let $f_p$ be the edge $e_1\cup \alpha_p$ in $G(\Sigma_p)$; it has ends $x_1$ and $p$. Also, let $\sigma^p=\sigma\cup \alpha_p$.  }
% The existence of obstructions in $G(\Sigma_p)$  is independent of how we draw $\alpha_p$ inside $F$. We will take advantage of this fact later on in the proof. 

Seeking a contradiction, suppose that each  $G(\Sigma_p)$ has an obstruction. 
Our next claim 
\junefif{gives two sufficient}  conditions on $p$ 
\junefif{that}
imply that all the obstructions in $G(\Sigma_p)$ contain $f_p$.

\setcounter{claim}{0}

\begin{claim}\label{middle_points_obs}
Let $p\in P$ 
be either one of $m_1,\ldots,m_n$ or not in $\sigma$. Then every  obstruction in $G(\Sigma_p)$  includes $f_p$. 
\end{claim}
\begin{proof}
Let $p\in P$ be such that  there is an obstruction $C$ in $G(\Sigma_{p})$ not including $f_{p}$. 

First, we show that $p$ is not a vertex in the middle of an edge in $W$. By contradiction, suppose  that $p=m_i$ for some $i\in \{1,...,n\}$.
\septtwotwo{Since $m_i$ is the only vertex 	 whose rotation in $G(\Sigma)$ differs from its rotation in $G(\Sigma_{m_i})$, $m_i\in V(C)$.}
 Consider the cycle $C'$ of $G(\Sigma)$ obtained by replacing the subpath $x_{i-1},m_i,x_{i}$ of $C$ by the edge $x_{i-1}x_{i}$. The 
\junefif{inside rotation of each vertex in $C'$}
 is the same as their rotation inside $C$. This shows that $C'$ is an obstruction in $G(\Sigma)$, a contradiction.  

Now suppose that $p$ is not in the middle of an edge in $W$. Then $C$ is a cycle in $G(\Sigma)$ and is not an obstruction in $G(\Sigma)$. The only vertex in $G(\Sigma_p)$ that has a rotation  that is different from its rotation in $G(\Sigma)$ is $p$. Therefore $p$ is a point in $C$ that is reflecting inside $C$ (witnessed by  two edges included in $\sigma^p$), and is  not reflecting in $C$ with respect to $G(\Sigma)$. Exactly one of the two witnessing edges is  in $G(\Sigma)$.  
So $p\in \sigma$. 
\end{proof}

More can be said about the obstructions in $G(\Sigma_p)$ for each point in $P$, but for this we need some terminology.  If we orient an edge $e$ in a plane graph, then the {\em sides} of $e$ are either the points near  $e$ that are to the right of $e$, or the points near $e$ to the left of $e$. 
%We are not interested on distinguishing between  left and right, but just to be aware that  $e$ has two sides.
Our next lemma shows that if $p\in P$, then all the obstructions in $G(\Sigma_p)$ include the same side of $f_p$ in its interior face. We defer its proof to Section \ref{sec:teclemmas} to keep the flow of the current proof. For the  convenience of the reader,  we provide all the hypotheses in the statement. 

\begin{restatable}{lemma}{twosidesedge}
\label{2SIDESEDGE}
Let $\Sigma$ be a set of strings in general position. Let $C_1$ and $C_2$ be obstructions  in $G(\Sigma)$ with $e\in E(C_1)\cap E(C_2)$. 
If  $C_1$ and $C_2$ include distinct  sides of $e$ in their interior faces,  then $G(\Sigma)$ has an obstruction not including $e$. 
\end{restatable}
\junefif{
The condition on the two cycles $C_1$ and $C_2$ containing distinct sides of $e$ \rev{implies that}   $e$ \rev{is}  incident with only interior faces of $C_1\cup C_2$. The perspective of the cycles being on distinct sides of $e$ is useful in the application, but what we really use in the proof of Lemma \ref{2SIDESEDGE} is that $e$ is not incident with the outer face of $C_1\cup C_2$.}

For each point $p\in P$, we will consider an  obstruction $C_p$  containing $f_p$; the choice of $C_p$ will be  more specific when $p\in \sigma$ (see below). 
\apriloneseven{
For $p\in P$, we orient $f_p$ from $x_1$ to $p$, so that we keep track of the side of $f_p$ contained in the interior of $C_p$. 
}

\apriloneseven{
Observe that $C_{x_1}$ contains the right of $f_{x_1}$ while $C_{x_{n-1}}$ contains the left of $f_{x_{n-1}}$ (here we use the fact that $F$ is bounded). This implies the existence of two consecutive vertices $x_{i-1}$, $x_i$ in $W-a$, such that}
 the interior of $C_{x_{i-1}}$ includes the right of $f_{x_{i-1}}$ and the interior of $C_{x_{i}}$ includes the left of $f_{x_{i}}$.

%It is important to know what side of $f_p$ our obstruction $C_p$ includes in its interior, so from now on we will assume $f_p$ is oriented from $x_1$ to \septtwotwo{$p$}. If $p\not\in \sigma$, then we let $C_p$ be any obstruction in $G(\Sigma_p)$ \septtwotwo{(Claim \ref{middle_points_obs} guarantees that $f_p\in E(C_p)$)}. 
%In case $p\in \sigma$,  we choose $C_{p}$ to be the  unique cycle included in the drawing of $\sigma^p$. Note that for $p=x_1$, the interior of $C_{p}$ includes the right of $f_p$, while for $p=x_{n-1}$, the interior of $C_p$ includes the left of $f_p$ (here we use the fact that the face $F$ is bounded). 

%\junefif{
%Our last observation implies that there are}
%two consecutive vertices $x_{i-1}$, $x_{i}$ in $W-a$ such that the interior of $C_{x_{i-1}}$ includes the right of $f_{x_{i-1}}$ and the interior of $C_{x_{i}}$ includes the left of $f_{x_{i}}$.

Without loss of generality, suppose that 
the interior of $C_{m_i}$ includes the left of $f_{m_i}$ (otherwise we reflect our drawing in a mirror).  To make the notation simpler, we let $x=x_{i-1}$ and $m=m_i$. We  may assume that $f_{m}$ is drawn near the left of $f_x$. 

The  next \octtwofour{claim} is the last ingredient to obtain a final contradiction.

\begin{claim}\label{claim:bothsides}
\septtwotwo{Exactly one of the following holds:}
   \begin{itemize}

       \item[(a)] \label{item1:bothsides} \septtwotwo{ $x\in \sigma$ and}  $G(\Sigma_m)$ has an obstruction containing $f_m$ whose interior includes a side that is distinct from the side included by $C_m$; \septtwotwo{or}
       \item[(b)]  \label{item2:bothsides} \septtwotwo{ $x\notin \sigma$ and} $G(\Sigma_x)$ has an obstruction containing $f_x$ whose interior includes a side of $f_x$ that is distinct from the side included by $C_x$.
    \end{itemize}

\end{claim}

%\begin{claim}
%The points  $x$ and $m$ are not included in $\sigma_0$.
%\end{claim}
\begin{proof}
\junefif{
\septtwotwo{First, suppose that $x\in \sigma$.} For (\ref{claim:bothsides}.a) we have two cases depending on whether \feba{$x_{i-1}x_i$} is an edge in $C_x$.}
\begin{enumerate}[label={\bf Case a.1}, wide, labelwidth=!, labelindent=0pt]
 \item {\em \feba{$x_{i-1}x_i$} is not in  $C_x$.}
 \end{enumerate}
 
 In this case we consider the cycle $C_m'$ obtained by replacing in $C_x$ the edge $f_x$ by the path  $P=(x_1$, $f_m$, $m$, $mx$, $x$). Since $x\in \sigma$, by the choice of $C_x$, all the edges in $C_x$ are in $\sigma^x$.
 Therefore all the edges in $C_m'$, with the possible exception of 
 $mx$, are in $\sigma^m$. Thus $C_m'$  is  an obstruction in $G(\Sigma_m)$. 
 
 It remains to show that the interior of $C_m'$ includes the right side of $f_m$. Note that $C_x\cup P$ consists of  three internally disjoint $x_1x$-paths, and because some points in $P$ are near the left side of $f_x$,  $P$ is in the outer face of $C_x$. The face of $f_x\cup P$ that is to the right of $f_m$ is included in the inner face $F$, so it is bounded. This implies the interior face of $C_m'$     includes the right of $f_m$. Since the interior of $C_m$ includes the left of $f_m$,   $C_m'$ and $C_m$ are  obstructions  including  distinct  sides of $f_m$.
 \begin{enumerate}[label={\bf Case a.2.}, wide, labelwidth=!, labelindent=0pt]
  \item {\em
\feba{$x_{i-1}x_i$} is in $C_x$. }
 \end{enumerate}

 In this case, $(x_1$, $f_{x}$, $x$, \feba{$xx_i$}, $x_{i})$ is a subpath 
of $C_x$.  We let $C_{m}'$ be the cycle obtained 
by replacing this path by $P=(x_1, f_m, m, mx_{i}, x_{i})$. Since $x\in \sigma$, the way we choose $C_x$ implies that all the edges in $C_x$ are in $\sigma^x$. So  all the edges in $C_m'$ are in $\sigma^m$, and  $C_m'$ is an obstruction. An  argument similar to  the one given in the previous case shows that the interior of $C_m'$ includes the right side of $f_m$. Thus the  interior of $C_m$ and $C_m'$ include distinct sides of $f_m$. 
 %\end{enumerate}

\junefif{
Turning to (\ref{item2:bothsides}.b),        \septtwotwo{ let us suppose that $p\not\in \sigma$.} We split the proof into two cases depending on whether $x$ is in $C_m$.}
\begin{enumerate}[label={\bf Case b.1.}, wide, labelwidth=!, labelindent=0pt]
\item {\em $x$ is in $C_m$.}
\end{enumerate}

First, we redraw $f_x$ and $f_m$ inside $F$ so that $f_x\cap f_m=\{x_1\}$. Let $T$ be the triangle bounded by $f_x$, $f_m$ and $xm$. The interior face of $T$ is to the left of $f_x$ and to the right of $f_m$.  
Consider   the $mx$-path $P$  of  $C_m$ that does not include the edge $f_m$. Since the interior face of $T$ is a subset of $F$, $P$ is drawn in the closure of  the exterior of $T$ (possibly \septtwotwo{$P=(m,mx,x)$}). 

Let $C$ be the simple closed curve bounded by $P\cup f_x\cup f_m$. We claim that the interior of $C$ is on the left of $f_x$. In the alternative, suppose that the interior of $C$ is on the right of $f_x$. Then $C'=P+ xm$ is a cycle of $G(\Sigma_m)$ including $f_x$ and $f_m$ in its interior. The $xx_1$-path $P'$ of $C_m$ that does not include $m$, is an arc connecting  $x_1$ to $x$ inside $C'$. Thus, $V(C')\subseteq V(C_m)$ and the closed disk bounded by $C'$ includes
$C_m$.  These two observations together imply that $C'$ has at most as many rainbows as $C_m$, and hence, $C'$ is an obstruction of $G(\Sigma_m)$ not including $f_m$.  Claim \ref{middle_points_obs} asserts that  all the obstructions in $G(\Sigma_m)$ include $f_m$, a contradiction.   Thus the interior of $C$ is on the left of $f_m$. 

From our last observation, it follows that $P'$ is an arc connecting $x_1$ and $x$ in the exterior of $C$. Because the interior of $C_m=P'\cup f_m \cup P$ is on the left of $f_m$, the interior of the cycle $C_x'=P'+f_x$ is on  the left of $f_x$. 

Now we show that $C_x'$ is an obstruction. 
Note that $V(C_x')\subseteq V(C_m)$ and that the closed disk bounded by $C_x'$ includes $C_m$. Then, every rainbow  in $C_x'$ is a rainbow in $C_m$, and hence $C_x'$ is an obstruction.  The cycles $C_x$ and $C_x'$ are obstructions including distinct sides of $f_x$ in their interiors, as claimed.
\begin{enumerate}[label={\bf Case b.2.}, wide, labelwidth=!, labelindent=0pt]
\item {\em $x$ is not in $C_m$.}
\end{enumerate}

In this case we let $C_x'$  be the cycle obtained by replacing the path $(x_1, f_m$, $m, mx_{i}$, $x_{i})$ in $C_{m}$ by the path $P= (x_1$, $f_x$, $x$, \marchtwo{$xx_i$}, $x_{i})$ in $G(\Sigma_x)$. Let $\alpha$ be the subarc of $P$ joining $x_1$ to $m$. As the points of $\alpha$ near $x_1$ are drawn on the left of $f_m$, and $\alpha$ is internally disjoint to \septtwotwo{$C_m$},  $\alpha$ connects $x_1$ and $m$ in the exterior of $C_m$. Since the  interior face of \septtwotwo{$\alpha\cup f_m$} is on the \marchtwo{left} of $f_x$, the interior face of $C_x'$ is on the left of $f_x$.

To show that $C_x'$ is an obstruction, note that the disk bounded by $C_x'$ includes $C_m$ and that $V(C_x')\setminus\{x\}\subseteq V(C_m)$. Thus all the rainbows  of $C_x'$ in $V(C_x')\setminus\{x\}$ are also rainbows in \feba{$C_m$}. The rotation of $x$ inside $C_x'$ is the list $(\feba{xx_i}, f_x)$, and, because  $x\notin \sigma$, $x$ is a rainbow   in $C_x'$, and is not a vertex of $C_m$. To compensate, we note that $m$ is a rainbow  in $C_m$ that is not in $V(C_x)$: if $m$ is not rainbow, both $f_m$ and $xx_{i}$ are included in $\sigma$, implying that $x\in \sigma$. This  shows that $C_x'$ has at most as many rainbows  as $C_m$. Thus $C_x'$ is an obstruction.   Again, the interiors of $C_x$ and $C_x'$ include distinct sides of $f_x$.
\end{proof}

By Claim \ref{claim:bothsides}, for some $p\in\{x,m\}$, $G(\Sigma_p)$ has obstructions including both sides of $f_p$ (and when $p=x$, we can guarantee that $p\notin \sigma$). Lemma \ref{2SIDESEDGE} implies that $G(\Sigma_p)$ has an obstruction not including $f_p$. Since either $p\notin \sigma$ or $p=m$, this last statement contradicts Claim \ref{middle_points_obs}. 
\end{proof}

%%%%%%%%%%%%%%%%
%%%%%%%%%%%

\begin{exterior_meeting}
Suppose that all the strings in $\Sigma$ have their ends \feba{on} the outer face of $G(\Sigma)$ and that all the ends have degree $1$ in $G(\Sigma)$. Then either all the strings are pairwise intersecting, and then $\Sigma$ can be extended to an arrangement of pseudolines, or we can extend two disjoint strings so that these strings  intersect without creating an obstruction.
\end{exterior_meeting}
\begin{proof}

\setcounter{claim}{0}

We start by considering a simple closed curve $\mathcal{O}$ containing all the ends of the strings in $\Sigma$, and that is otherwise disjoint 	from $\bigcup\Sigma$. We  construct this curve by connecting  each pair of vertices  with degree 1 that are consecutive in the boundary walk of the outer face. 
To connect these pairs  we use an arc whose interior is included in the outer face, near the portion of the boundary walk between the two vertices. 

\feba{
Suppose $\sigma_1$, $\sigma_2$ are two disjoint strings in  $\Sigma$.}
For  $i=1,2$, let  $a_i$, $b_i$ be  the  ends of $\sigma_i$. 
 Since $\sigma_1$ and $\sigma_2$ do not intersect inside $\mathcal{O}$, their ends do not alternate as we traverse $\mathcal{O}$ in counterclockwise order. We may assume, by relabeling if necessary,  that the ends occur in the order $a_1$, $b_1$, $b_2$, $a_2$. 
 
 We extend the $a_i$-ends of $\sigma_1$ and $\sigma_2$  so that they meet in a point $p$ in the outer face. We  do this extension so that  the two added segments are in the outer face,  and\feba{,} more \feba{importantly}, so that the interior face of the simple closed curve bounded by the added segments and the $a_2a_1$-arc in $\mathcal{O}$ not containing $\{b_1,b_2\}$, does not include the inner  face of $\mathcal{O}$. In Figure \ref{how_to_extend_f_e} we show the right and wrong way to extend, respectively.
 
 \begin{figure}[ht]
 	
    \centering
    
    \includegraphics[scale=0.6]{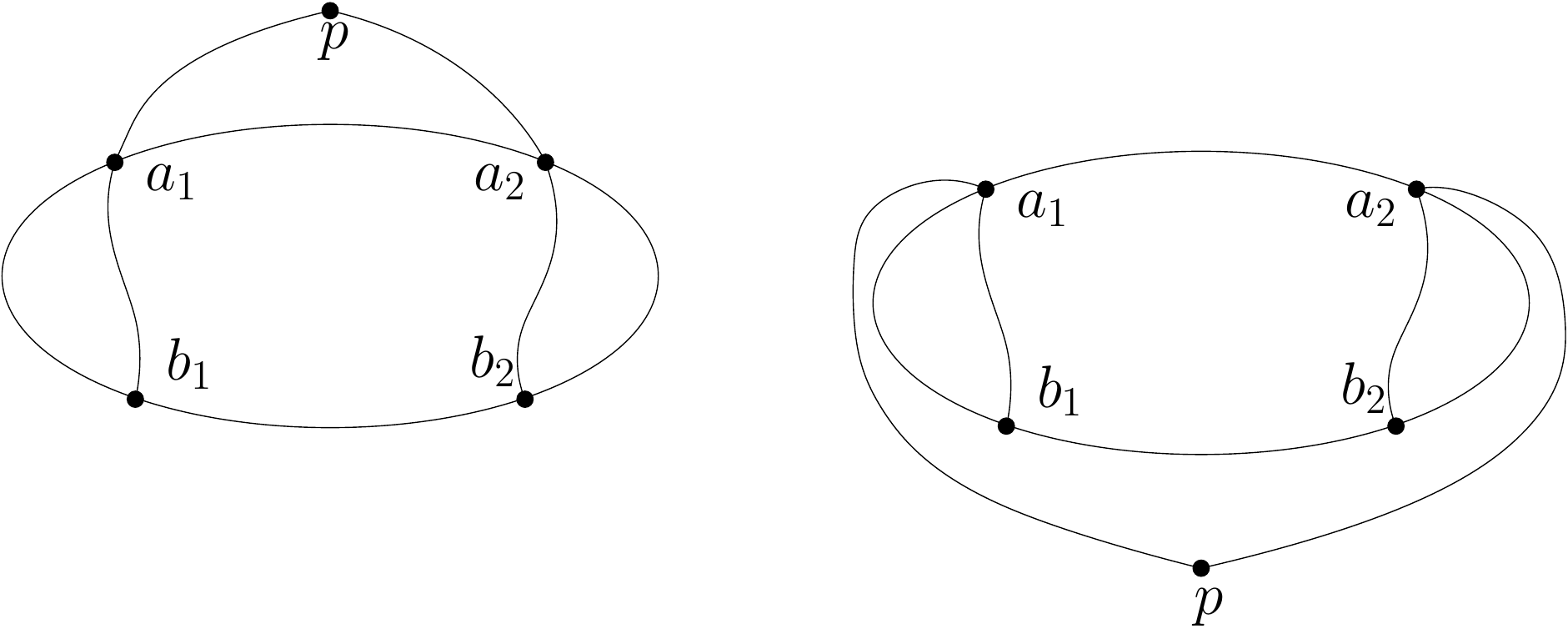}
    \caption{The right and wrong way to extend  in the Exterior-Meeting Step.}
 \label{how_to_extend_f_e}
\end{figure}

\septtwotwo{We denote  the new set of strings obtained as above \ther{by $\Sigma'$}.}
To show that $\Sigma'$ has no obstruction, we consider a cycle $C$
in $G(\Sigma')$. If $C$ does not contain $p$, then $C$ is a cycle in $G(\Sigma)$, and so is  not an obstruction in $G(\Sigma')$. Now suppose that $p$ is in $ C$. 

The idea is to find three rainbows  in $C$. To get the first one, we consider   the path $P_1$  obtained by  traversing $C$, starting at $p$, continuing along the path induced by $\sigma_1$, and stopping  just before we reach a first vertex not in $\sigma_1$. Let $c_1$ be the last vertex in $P_1$,  and  let $d_1$ be the \septtwotwo{neighbour} of $c_1$ in $C$ that is not in $P_1$.  

\begin{claim}\label{claim:one_rainbow_in_disk}
The cycle $C$ has a rainbow  included  in the disk $\Delta_1$ bounded by $\sigma_1$ and the $a_1b_1$-arc of $\mathcal{O}$ not containing $a_2$.
\end{claim}
\begin{proof}
The vertex $d_1$ is in one of the two bounded faces of $\mathcal{O}\cup \sigma_1$.
Suppose  that $d_1$ is in the  face $F$ that is bounded by $\sigma_1$ and the $a_1b_1$-arc of $\mathcal{O}$ containing $a_2$ and $b_2$. 
The rotation at $c_1$ inside $C$ does not include two edges in the same string $\sigma$, as otherwise $\sigma$ and 
\junefif{
$\sigma_1$} tangentially intersect at $c_1$. Therefore, when $d_1\in F$,  $c_1$ is  a rainbow  of $C$ in $\Delta_1$. 

\junefif{Now suppose that $d_1$ is in $\Delta_1$.}
\junefif{
Let $P_1'$ be the path \septtwotwo{of $C$} starting at $c_1$ and the edge $c_1d_1$, and ending at the first vertex we encounter that is in}
\junefif{$\sigma_1$.} 
The cycle $C'$ enclosed by $P_1'$ and 
\junefif{
$\sigma_1$} is not an obstruction, so it has at least three rainbows.
The vertices in $C'-V(P_1')$ are reflecting inside $C'$ because their rotations inside $C'$ contain two edges in $\sigma$.  Hence  at least one internal vertex of \ther{$P_1'$} is a rainbow in \marchtwo{$C'$}. This  vertex is also \ther{a} rainbow in $C$, and is included in $\Delta_1$. 
\end{proof}

 Considering  $\sigma_2$ instead of $\sigma_1$, Claim \ref{claim:one_rainbow_in_disk} yields a second rainbow   in $C$ inside an  analogous disk $\Delta_2$.  The third rainbow   is $p$, showing that $C$ is not an obstruction.
\end{proof}
Since the Disentangling Step, Face-Escaping Step and Exterior-Meeting Step can be performed without creating new obstructions, either\feba{:} one of these steps can be performed to increase the number of pairwise intersecting strings in $\Sigma$\feba{;} or the strings in $\Sigma$  are pairwise intersecting and all of them have their ends in the outer face, which implies that $\Sigma$
can 
\junefif{be}
extended to an arrangement of pseudolines.
\end{proof}

\section{Proof of Lemmas \ref{2SIDES} and \ref{2SIDESEDGE}}\label{sec:teclemmas}
We deferred the proofs of   Lemmas \ref{2SIDES} and \ref{2SIDESEDGE}, both essential in the proof of Theorem \ref{MAIN}, to this section.

Our next observation follows immediately  from the definition of  rainbow, and it will be repeatedly used in the  next proofs.

\begin{usefulfact}
Let $\Sigma$ be 
set of strings in general position. 
Let $v$ be a vertex that is in both the cycles $C$ and $C'$ of $G(\Sigma)$ such that the rotation at $v$ inside $C$  includes  the rotation at $v$ inside $C'$. If $v$ is a rainbow  in $C$, then 
\junesix{v}
is a rainbow in $C'$.
\end{usefulfact}

\junesix{
Recall that a {\em near-obstruction} at $v$ is a cycle
$C$ (in the underlying graph of a set of strings) that has precisely three
rainbows, one of which is $v$. In Figure \ref{fig:local_near_obstructions}, we depict (up to symmetries) how two near-obstructions may intersect at $v$. In each of the nine diagrams, $v$ is represented as a black dot, while the interiors of the near-obstructions are represented as dotted and dashed lines. In our next lemma,  we will consider two near-obstructions at $v$ that intersect only as  in the last three diagrams, where
 \feba{every} small open disk centered at $v$ is included in the \feba{union}  of the disks bounded by the two near-obstructions. In the statement, an equivalent description is given in terms of the local rotation at $v$.

}

\begin{figure}[ht]
 	
    \centering
    
    \includegraphics[scale=0.7]{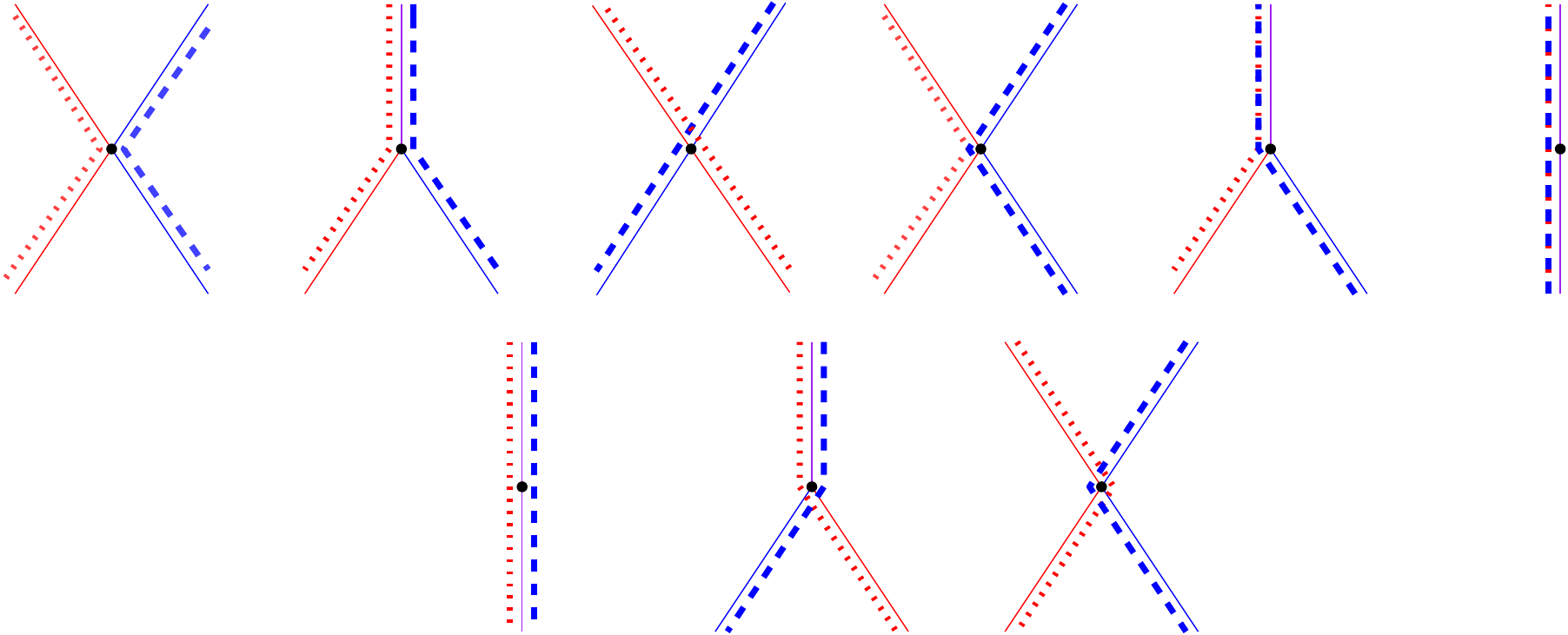}
    \caption{Two near obstuctions at $v$.}
\label{fig:local_near_obstructions}
\end{figure}

\twosides*

\setcounter{claim}{0}
\begin{proof}
In order to obtain a contradiction,  suppose that $G(\Sigma)$ has no obstructions and that it contains such cycles $C_1$, $C_2$.
\junesix{
The conditions on the rotation at $v$ imply that every edge incident with $v$ is in the interior of either $C_1$ or $C_2$. Thus, $v$ is not incident with the outer face of $C_1\cup C_2$. }

Our next goal is to show that $C_1\cap C_2$ has at least two vertices. 
\feba{If  $e$ is an edge of $C_1$ incident with $v$, then either $e$ is an edge of $C_2$ or $e$ is inside $C_2$. In the former case $|V(C_1)\cap V(C_2)|\geq 2$, thus we may assume that both edges of $C_1$ incident with  $v$ are inside $C_2$. If $v$ is the only vertex in $V(C_1)\cap V(C_2)$, then $C_1-v$ is in the interior of $C_2$, and hence the edges of $C_2$ incident with $v$ are not in the rotation at $v$ inside $C_1$, a contradiction. Thus  $|V(C_1)\cap V(C_2)|\geq 2$.}
 It follows that $C_1\cup C_2$ is 2-connected; in particular, its outer face is bounded by a cycle $C_{out}$.

%We begin by showing that the outer face of the graph  $C_1\cup C_2$ is cycle and that $v$ is not on this outer cycle. Since $C_1\cup C_2$ is planar, it is enough to show that  is 2-connected to guarantee that the outer boundary is a cycle.  Each vertex in $C_1\cup C_2$ belongs  to one of the cycles $C_1$ or $C_2$, so if these two cycles have at least two vertices in common we are certainly done. The vertex $v$ is in both cycles, and the conditions for the rotations at $v$ imply that all the faces of $C_1\cup C_2$ incident to $v$ are either inside $C_1$ or inside $C_2$. This shows that $v$ is not in the boundary of the outer face of $C_1\cup C_2$. Either  $C_1$ and $C_2$ share an edge incident to $v$, or we can find a path of $C_1$, starting at $v$, ending at a vertex in $C_2$, and else disjoint from $C_2$. The non-$v$ end of this path is another vertex in $V(C_1)\cap V(C_2)$. Thus $C_1\cup C_2$ is 2-connected, and $v$ is not in the outer cycle $C_{out}$. 

The Useful Fact applied to  $C=C_{out}$ and to each $C'\in \{C_1, C_2\}$, shows that every  vertex that is a rainbow in $C_{out}$ is also a rainbow in each of the cycles in $\{C_1, C_2\}$ containing it.
\junesix{
By assumption, $C_{out}$ is not an obstruction, so it has at least three rainbows.
\junefif{The \feba{preceding} two sentences imply that we may choose}
 the labelling such that two of them, say $p$ and $q$, are also rainbows in $C_1$. Neither $p$ nor $q$ is $v$ and $C_1$ is a near-obstruction. Thus, $p$ and $q$ are the only rainbows of $C_{out}$ that are in $C_1$. 
}

Since $v\notin V(C_{out})$, $C_1$ has a subpath 
\junesix{$P_v$}
  containing $v$  in which only the ends of 
\junesix{$P_v$}
are in $C_{out}$. Since $v$ is not in the outer face of $C_{out}$, \junesix{$P_v$}
is included in the inner face of $C_{out}$. 
 We let $u$ and $w$ be the ends of 
\junesix{$P_v$,}
and  let   
\junesix{$Q_{out}^1$, $Q_{out}^2$}
be the 
\junesix{$uw$-paths}
of $C_{out}$.
\junesix{
The cycle  $C_1$ is inside one of the two disks bounded by 
\junesix{$P_v$}
and one of  $Q_{out}^{1}$ and $Q_{out}^{2}$.} By symmetry, we may assume that $C_1$ is included in the disk bounded by 
\junesix{$Q_{out}^1\cup P_v$.}
In this case 
\junesix{$Q_{out}^2$}
is a subpath of $C_2$.

Our desired contradiction will be obtained by finding three rainbows  in $C_2$  distinct from $v$. The first is \junefif{relatively} easy to find: \junefif{if} 
\junesix{$C_1-(P_v)$}
is the   $uw$ path  in $C_1$ distinct from  
\junesix{$P_v$,}
we consider the cycle 
\junesix{$(C_1-(P_v))\cup Q_{out}^2$.}
The disk bounded by  
\junesix{$(C_1-(P_v))\cup Q_{out}^2$}
contains the one bounded by $C_1$. Then the Useful Fact applied to 
\junesix{$C=(C_1-(P_v))\cup Q_{out}^2$}
and $C'=C_1$, implies that each vertex in  
\junesix{$C_1-(P_v)$}
that is rainbow  in 
\junesix{$(C_1-(P_v))\cup Q_{out}^2$}
is also rainbow in $C_1$.  Since $C_1$ has at most two rainbows in 
\junesix{$C_1-(P_v)$,}
namely $p$ and $q$,  
\junesix{$(C_1-(P_v))\cup Q_{out}^2$}
must have a third rainbow  $r_1$  in the interior of 
\junesix{$Q_{out}^2$.}
The \junesix{interiors of the}  disks bounded by $C_2$ and 
\junesix{$(C_1-(P_v))\cup Q_{out}^2$}
are on the same side of 
\junesix{$Q_{out}^2$;}
\junesix{thus}
$r_1$ is a rainbow for $C_2$.

To find another rainbow   in $C_2$,  consider the edge $e_u$ of $C_2$ incident to $u$ and not in 
\junesix{$Q_{out}^2$.}
We claim that either $u$ is a rainbow in $C_2$ or that $e_u$ is not included in the
\junefif{closed}
disk bounded by 
\junesix{$P_v\cup Q_{out}^2$.}
Looking for a contradiction, suppose \junesix{that} $u$ is reflecting in $C_2$ and that $e_u$ is included in the disk. Then  we can find two edges in the rotation at $u$, included in the disk bounded by $P_v\cup Q_{out}^2$,
that belong to  the same string $\sigma$. The vertex $u$ is not reflecting in $C_1$, as else, we would find another pair of edges in the rotation at $u$ inside 
\junesix{$Q_{out}^1\cup P_v$,}
and included in a different string $\sigma '$; in this case, $\sigma$ and $\sigma '$ tangentially intersect at $u$, a contradiction. Therefore $u$ is a rainbow in $C_1$, so 
\junesix{$u$}
is one of $p$ and $q$. This implies that $u$ is a rainbow in $C_{out}$, and hence, a rainbow in $C_2$, a contradiction. 

\junefif{
If $u$ is a rainbow in $C_2$, then this is the desired second one. Otherwise, the preceding paragraph shows that $e_u$ is not in the closed disk bounded by $P_v\cup Q_{out}^2$. In this latter case, $e_u$ is in a path $P_u$ that starts at $u$, ends at $u'$ in $P_v$ and is otherwise disjoint from $P_v$.}

\junefif{
Note that $u'\neq w$, as otherwise $C_2=P_u\cup Q_{out}^2$ and we have the contradiction that $v$ is not in $C_2$. Let $C_u$ be the cycle consisting of $P_u$ and the $uu'$-subpath $uP_vu'$ of $P_v$.}

\begin{claim} \label{claim:smaller_disk}

If $P_u$ does not have a rainbow of $C_u$ in its interior, then:
\begin{itemize}
\item[(a)] $C_u$ and $C_2$ are near-obstructions at $v$ satisfying the conditions in Lemma \ref{2SIDES};
and 
\item[(b)] the closed disk bounded by the outer cycle of $C_u\cup C_2$ contains fewer vertices than the disk bounded by $C_{out}$.
\end{itemize} 

%At least one of the interior vertices in $P_u$ is a rainbow in $C_u$, and consequently, a rainbow in $C_2$.
\end{claim}
\begin{proof}

Suppose that all the 
 \junesix{rainbows}
  of $C_u$ are located in 
  \junesix{$u P_v u'$.}
   Since $C_u$ is not an obstruction, at least one of them is an interior vertex of 
   \junesix{$u P_v u'$.}
   Each vertex in the interior of 
   \junesix{$u P_v u'$}
   that is a rainbow in $C_u$, is also a rainbow in $C_1$. As $v$ is the only vertex in the interior of 
   \junesix{$P_v$}
   that is a rainbow in $C_1$,  $v$ is the only  rainbow of $C_u$ that is in the  interior  of 
   \junesix{$u P_v u'$.}
   \junefif{ Since $C_u$ is not an obstruction, $u$, $u'$ and $v$ are the only rainbows of $C_u$, and $C_u$ is a near-obstruction at $v$. The rotation at $v$ inside $C_u$ is the same as inside $C_1$, so $C_u$ and $C_2$ satisfy the conditions in Lemma \ref{2SIDES}.}

\junefif{Let $C_{out}'$ be the outer cycle of $C_u\cup C_2$.}
Since $C_{u}\cup C_2\subseteq C_1\cup C_2$, the exterior of $C_{out}$ is included in the exterior of $C_{out}'$.  This shows that the disk bounded by $C_{out}$ includes the disk bounded by $C_{out}'$.

%To prove that the inclusion is proper, we will see that one of $p$ and $q$ is not in the disk bounded by $C_{out}'$.

If both $p$ and $q$ are in $C_2$, then $p$, $q$ and $r_1$ are  rainbows  in $C_2$, and also distinct from $v$, contradicting that $C_2$ is a near-obstruction for $v$. Thus, we may assume $p\notin C_2$. 
 Then  $p$ is not in 
 \junesix{$ P_u\subseteq C_2$}
 and, since 
  $p$ is not an interior vertex  of 
  \junesix{$P_v$,} 
  \junesix{$p\notin V(C_u)$.}
  \junefif{
  Since $p$ is in $C_{out}$, and $p$ is not in $C_u\cup C_2$, $p$ is in the outer face of $C_u\cup C_2$.} 
  \septtwotwo{Then $p$ is in the disk bounded by $C_{out}$ but not by $C_{out}'$, as required.}
\end{proof} 

\junefif{
The proof of the existence of the additional two rainbows in $C_2$ is by induction on the number of vertices in the closed disk bounded by $C_{out}$. If $u$ is not a rainbow in $C_2$ and $P_u$ does not have a rainbow of $C_2$ in its interior, then Claim \ref{claim:smaller_disk} implies $C_u$ and $C_2$ make a smaller instance and we are done. Thus, we may assume one of them yields the next additional rainbow.}

\junefif{
In the same way, either the induction applies or the last rainbow comes by considering the edge of $C_2-Q_{out}^2$ incident with $w$. It follows that $v$, $r_1$, and these two other vertices are four different rainbows in $C_2$, contradicting the fact that $C_2$ is a near-obstruction. }
%\junesix{From all our previous discussion,}
%either  $u$ is a rainbow in $C_2$, or we find one vertex in the interior of $P_u$  that is a rainbow in $C_2$.
%Likewise, either $w$ is a rainbow in $C_2$, or we can define an analogous path $P_w$  that has a vertex in its interior that is a rainbow in $C_2$.  The paths $P_u$ and $P_w$ are distinct, as else, 
%\junesix{$C_2=P_u\cup Q_{out}^2$, and $C_2$}
%would not contain $v$. Therefore, $C_2$ has three rainbow vertices  distinct from $v$, contradicting that $C_2$ is a near-obstruction. 
\end{proof}

Although the statements and  proofs of Lemmas \ref{2SIDES} and \ref{2SIDESEDGE} are  similar, some subtle differences make 
\junefif{it}
hard to find  a  statement  encapsulating both results. For instance,  Lemma \ref{2SIDESEDGE} assumes that  $G(\Sigma)$ has  obstructions, while finding an obstruction is the conclusion of Lemma \ref{2SIDES}. We sketch the proof of Lemma \ref{2SIDESEDGE}, emphasizing such differences. \junesix{ It would be interesting to find  a common  theory behind these two lemmas.}

\twosidesedge*

\begin{proof}[Sketch of the proof]
We start assuming that such cycles exist and that every obstruction  includes $e$. 

\junesix{
By assumption, $C_1\cap C_2$ has at least two
vertices and, therefore, $C_1\cup C_2$ is $2$-connected. Thus, its outer face is
bounded by a cycle $C_{out}$.}

The Useful Fact shows  that every rainbow  in $C_{out}$ is a rainbow in each of the cycles 
\junefif{$C_1$ and $C_2$}
containing it. 

\junesix{Since $C_1$ and $C_2$ include
different sides of $e$, it follows that $e$ is not in $C_{out}$.
Therefore  $C_{out}$  is not an obstruction.}
\junesix{Thus, $C_{out}$}
has at least three rainbows, and by our previous observation, 
\junesix{we may choose the labelling such that two of them, say $p$ and $q$, are also rainbows in $C_1$.}
Because $C_1$ is an obstruction,  $p$ and $q$  are the only rainbows in  $C_1$.

 Then  $C_1$ has a subpath 
 \junesix{$P_e$}
 of  containing $e$ and in which only the ends $u$ and $w$ of 
 \junesix{$P_e$}
 are in $C_{out}$. 
Let 
\junesix{$Q_{out}^1$}
and 
\junesix{$Q_{out}^2$}
be the $uw$-paths of $C_{out}$. We may assume 
\junesix{that}
$C_1$ is drawn in the disk bounded by 
\junesix{$Q_{out}^1\cup P_e$.} 

Let 
\junesix{$C_1- (P_e)$}
be $uw$-path in $C_1$ that is not 
\junesix{$P_e$.}
Note that $p$ and $q$ are the only vertices  in  
\junesix{$C_1- (P_e)$}
that are rainbows in 
\junesix{$(C_1- (P_e))\cup Q_{out}^2$.}
Since 
\junesix{$(C_1- (P_e))\cup Q_{out}^2$}
is not an obstruction, the interior 
\junesix{$Q_{out}^2$}
has a vertex $r_1$ that is a rainbow of  
\junesix{$(C_1- (P_e))\cup Q_{out}^2$.}
This vertex $r_1$  is also a rainbow of $C_2$. 

Let $e_u$ be the edge incident to $u$ in $C_2$ that is not in 
\junesix{$Q_{out}^2$.}
As we did in Lemma \ref{2SIDES}, we can show  that either $u$ is a rainbow in  $C_2$ or that $e_u$ is  not included in the disk bounded by 
\junesix{$P_e\cup Q_{out}^2$.} We  assume the latter situation, as in the former we found \junefif{our desired second  rainbow} in $C_2$.

\junefif{
Let $P_u$ be the subpath of $C_2$ starting at $u$, continuing on $e_u$, and ending on the first vertex 
$u'\in V(P_e)\cap V(C_2)$
distinct from $u$. 
Note that $u'\neq w$, as otherwise $C_2=P_u\cup Q_{out}^2$ and we have the contradiction that $e$ is not in $C_2$. Let $C_u$ be the cycle consisting of \septtwotwo{$P_u$} and the $uu'$-subpath $uP_eu'$ of $P_e$. }

\junefif{
We claim that 
either
$P_u$ has an interior vertex that is a rainbow in $C_2$ or that there is a pair of cycles $C_1'$ and $C_2'$ satisfying the conditions in Lemma \ref{2SIDESEDGE}, but with fewer vertices in the closed disk bounded by the outer cycle of $C_1'\cup C_2'$ than in the disk bounded by $C_{out}$.}

Suppose that none of the interior vertices in  $P_u$ is a rainbow in $C_2$. 
Because the interior of $P_e$
has no vertices that are rainbows in $C_1$ 
\junesix{(as $p$ and $q$ are the only rainbows of $C_1$),}
the interior of 
\junesix{$uP_eu'$}
has no vertices that are rainbows in $C_u$.
\junefif{ Therefore $C_u$ is an obstruction, 
and  $C_1'=C_u$ and $C_2'=C_2$ is a pair of obstructions including both sides of $e$. 
As $C_u\cup C_2\subseteq C_1\cup C_2$, the closed disk bounded
by the $C_{out}$ contains the closed disk bounded by the outer cycle of $C_u\cup C_2$. Not both of $p$ and $q$ are in the outer cycle of $C_u\cup C_2$, as both $p$ and $q$ would be part of $C_2$, concluding that $C_2$ has three rainbows $p$, $q$ and $r_1$, and contradicting that $C_2$ is an obstruction.}

\junefif{
From the previous paragraph, either $C_u$ and $C_2$ is a smaller instance, and we are done by induction on the number of vertices in the closed disk bounded by $C_{out}$, or we found our second rainbow of $C_2$ in the interior of $P_u$.}

\junefif{
In the same way, either the induction applies or the last rainbow comes by considering an edge of $C_2-Q_{out}^2$ incident with $w$. It follows that $r_1$, and these other two vertices are three different rainbows in $C_2$, contradicting that $C_2$ is an obstruction.}
\end{proof}

\section{Finding obstructions in polynomial time}\label{sec:find_obs}

In this section we describe a polynomial-time algorithm that determines whether
\junefif{a}
set of strings has an obstruction. We will assume that our input is the underlying plane graph $G(\Sigma)$   of a set $\Sigma$ of simple strings  in general position, and  \septtwotwo{  that every string in $\Sigma$ is identified as a path in $G(\Sigma)$ (see notation below). }
 
\septtwotwo{
 The key idea behind the algorithm is simple: \feba{either} find an obstruction in the outer boundary of $G(\Sigma)$ or find a vertex  in the outer boundary whose removal reduces our problem into a smaller instance.  }
 
 \septtwotwo{ We start \ther{by} describing the vertex removal operation. 
 Suppose that $x$  is a vertex \feba{of $G(\Sigma)$}  incident to the outer face of $G(\Sigma)$.
For each $\sigma\in \Sigma$, we consider the path $P_\sigma$ of $G(\Sigma)$ representing $\sigma$. Let  $P_\sigma-x$ be the plane graph obtained from $P_\sigma$ by removing $x$  and the edges of $P_\sigma$ incident to $x$  \feba{(}if $x\notin P_\sigma$, then $P_\sigma-x=P_\sigma$\feba{)}.    Each component of  $P_\sigma-x$ is either a vertex that represents an end of $\sigma$, or   a string. Let $S_{\sigma,x}$ be the set of string components of $P_\sigma-x$ and let $\Sigma-x=\bigcup_{\sigma\in \Sigma}S_{\sigma,x}$. Note that $G(\Sigma-x)$ can be obtained from $G(\Sigma)$ by removing  $x$ and the edges incident to $x$, and then suppressing the \octtwofour{degree-$2$} vertices whose incident edges belong to the same string in $\Sigma$, as well as removing remaining \octtwofour{degree-$0$} vertices (Figure \ref{fig:need_to_suppress} \octtwofour{illustrates} this process). 

}

\begin{figure}[ht]

    \centering
    
    \includegraphics[scale=0.8]{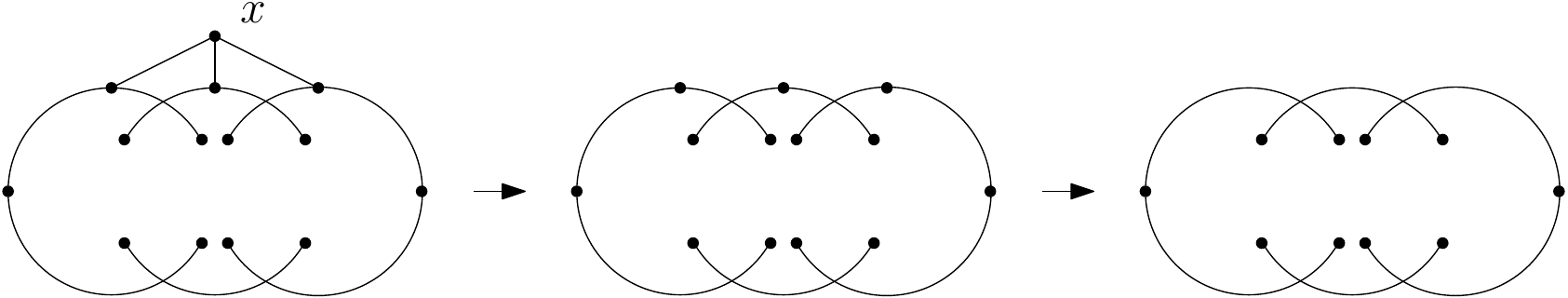}
    \caption{From $\Sigma$ to $\Sigma-x$.}
	 	\label{fig:need_to_suppress}
\end{figure}

\septtwotwo{The next lemma is the key property used in the algorithm.} % We will use use it to prove  correctness of our algorithm.

\begin{lemma} \label{correctness}
\septtwotwo{
Let  $\Sigma$ be a set of simple strings  in general position and let $x$ be a vertex incident with the outer face. Then there is a $1-1$ correspondence between the obstructions in $G(\Sigma)$  not containing $x$ and the obstructions of $G(\Sigma-x)$.  }
\octtwofour{Moreover, corresponding obstructions are the same simple closed curve.}

\end{lemma}

\begin{proof}
\septtwotwo{
In general, there is a natural correspondence between cycles in $G(\Sigma)$ not containing $x$ and cycles in $G(\Sigma-x)$: 
\feba{if}  $C$  \octtwofour{is} a cycle in $G(\Sigma)$ not containing $x$, then    every  edge of $C$ is not incident with $x$, and hence every edge  is part of a string in $\Sigma-x$. Thus, there is a cycle $C'$ in $G(\Sigma-x)$ that represents the same simple closed curve as $C$. Conversely, each cycle $C'$ in $G(\Sigma-x)$ is a simple closed curve in $\bigcup (\Sigma-x)\subseteq \bigcup\Sigma$, and hence, there is a cycle $C$ in $G(\Sigma)$  representing  the same simple closed curve as $C'$.

To complete the proof
 it is enough to show that any two  cycles $C$, $C'$ that correspond as above have the same   rainbows. 
Since $G(\Sigma-x)$ is obtained from suppressing and removing vertices in a subgraph of $G(\Sigma)$,   $V(C')\subseteq V(C)$.  Thus,  $V(C)\setminus V(C')$ 
 consists of suppressed \marchtwo{and removed} vertices in the process of converting $G(\Sigma)$ into $G(\Sigma-x)$. Since $x\notin V(C)$, if $v\in V(C)$ is suppressed, then   the two  edges of $C$ incident to $v$ belong to the same string in $\Sigma$. \octtwofour{Therefore}, none of the vertices in $V(C)\setminus V(C')$ is a rainbow in $C$. 
 
 Every rainbow in $C$ is also a rainbow in $C'$ because every  two edges of $G(\Sigma-x)$ that are included in distinct strings of $\Sigma$ are also included in  distinct strings in $\Sigma-x$.

Conversely, \octtwofour{suppose that} $v\in V(C)\cap V(C')$ is reflecting \octtwofour{in} $C$. Let $\sigma\in \Sigma$ be a string including two  edges of $G(\Sigma)$ in the  rotation at $v$ inside $C$. Since $x$ is drawn in the exterior of $C$, these two edges are part of the same string in $\Sigma-x$, and hence $v$ is reflecting inside $C'$. 

Therefore every rainbow of $C'$ is a rainbow of $C$, and thus, $C$ and $C'$  have the same   rainbows. 
}
\end{proof}

A vertex in $G(\Sigma)$ is an
\newcommand{\outr}{outer-rainbow }{\em \outr} if it is in the outer boundary and all the edges in its rotation belong to different strings. Note that every \outr  is a rainbow  for all the cycles in $G(\Sigma)$ that contain it.

 An {\em outer cycle} is a cycle of $G(\Sigma)$ that has all its edges incident to the outer face of $G(\Sigma)$. 
 \junefif{
 For any graph $G(\Sigma)$, a {\em block} of $G(\Sigma)$ is a maximal connected subgraph of $G(\Sigma)$ with no cut-vertex. If $G(\Sigma)$ is connected with at least two vertices, then each block is either an edge or is $2$-connected. In the latter case, the outer face of the block is bounded by a cycle of the block.}

We  find obstructions by solving 
\junefif{an}
auxiliary problem: finding  obstructions including one or two fixed outer-rainbows.  
The next subroutine \septtwotwo{(Algorithm \ref{alg:1})} describes how to find an obstruction containing two fixed outer-rainbows. Below we discuss its correctness.

\vspace{0.5cm}
\begin{algorithm}[H]
  \SetAlgoLined
  \KwData{%Underlying plane graph
   $G(\Sigma)$ and two \junefif{outer-rainbows}  $x$ and $y$.}
  \KwResult{\junefif{Either an obstruction containing $x$ and $y$ or that no such obstruction exists.} }

  \Repeat{$V(G(\Sigma))=\{x,y\}$ }{
  	\If{there is no cycle containing $x$ and $y$}
    { \Return $G(\Sigma)$ has no obstruction containing $x$ and 								$y$;
            	}
      Find the outer cycle $C$  containing 		$x$ and $y$\;
		\While{\septtwotwo{$C$ is not the outer boundary of $G(\Sigma)$}}{
		\septtwotwo{Pick $w\in V(G(\Sigma))\setminus V(C)$ incident with the outer face\;
		$\Sigma\longleftarrow \Sigma-w$}
		
		}

        \eIf{$C$ has \junefif{a} rainbow  $z\notin\{x,y\}$ in $G(\Sigma)$}
        {\septtwotwo{$\Sigma\longleftarrow \Sigma-z$;} }
        {\Return $C$;}

  }
 { \Return   $G(\Sigma)$ has no obstruction containing $x$ and $y$. }
  
  \caption{Finding obstructions through two fixed outer-rainbows.}
  \label{alg:1}
\end{algorithm}
\vspace{0.5cm}

\septtwotwo{To see that Algorithm \ref{alg:1} is correct, observe that} when Step 2 does not apply, then Step 5 can be performed:
\marchtwo{if} there is a cycle containing $x$ and $y$, then, as $x$ and $y$ are incident to the outer face of  $G(\Sigma)$, the outer boundary  of the block containing $x$ and $y$ is an outer cycle $C$ containing $x$ and $y$. \septtwotwo{ Every obstruction $\mathcal{C}$ through $x$ and $y$ is \octtwofour{drawn in the closed disk bounded by $C$.}
 Lemma \ref{correctness} guarantees that if we remove a  vertex in the outer boundary that is not in $C$ (Step 7) and we update $\Sigma$ (Step 8), then $\mathcal{C}$  (or more precisely, the cycle in the new $G(\Sigma)$ \octtwofour{that is} the same simple closed curve as $\mathcal{C}$) is  an obstruction through $x$ and $y$. }

\septtwotwo{
In Step 10, if $x$ and $y$ are the only rainbows of $C$, then $C$ is an obstruction returned in Step 13. Else, $C$ has  a rainbow $z\notin \{x,y\}$. Any obstruction $\mathcal{C}$ through $x$ and $y$ does not contain $z$, and hence removing $z$  and updating $\Sigma$ (Step 11) does not change the fact that $\mathcal{C}$ is an obstruction in the new $G(\Sigma)$. This algorithm terminates as the number of vertices in $G(\Sigma)$ is always decreasing.
}

We now turn to Algorithm \ref{alg:2}, used as subroutine in the main algorithm. 
 Its correctness again easily follows from   Lemma \ref{correctness}.

\vspace{0.5cm}
\begin{algorithm}[H]
  \SetAlgoLined
  \KwData{$G(\Sigma)$ and an \outr vertex $x$.}
  \KwResult{Either an obstruction containing $x$ or that no such obstruction exists. }

  \Repeat{$V(G(\Sigma))=\{x\}$}{
  	\If{there is no cycle containing $x$}
    {\Return $G(\Sigma)$ has no obstruction containing $x$;
            	}
\septtwotwo{Find an outer cycle $C$ containing $x$}\;

        \eIf{\septtwotwo{$C$ has an \outr  $y\neq x$}}
        { Run Algorithm \ref{alg:1} on  $(G(\Sigma),x,y)$\;
        	\If{$G(\Sigma)$ has an obstruction $D$ including 						$x$ and $y$,
            	}
            	{
            		\Return $D$;}
        
       \septtwotwo{ $\Sigma\longleftarrow  \Sigma-y$;} }
        {\Return $C$;}

  }
   \Return{ $G(\Sigma)$ has no obstruction containing $x$.}
   
  \caption{Finding obstructions through a  fixed outer-rainbow.}
  \label{alg:2}
\end{algorithm}
\vspace{0.5cm}

Finally we present the algorithm to find obstructions, whose correctness  also relies  on Lemma \ref{correctness}. 

\hspace{0.5cm}

\begin{algorithm}[H]
  \SetAlgoLined
  \KwData{$G(\Sigma)$.}
  \KwResult{Either finds an obstruction or that no such obstruction exists. }

  \Repeat{$G(\Sigma)=\emptyset$}{
  	\If{$G(\Sigma)$ has no cycles}
    {\Return $G(\Sigma)$ has no obstructions;
            	}
  
      \septtwotwo{	Find an outer cycle $C$\;
  }

           \If{ $C$ has no rainbows}{ \Return $C$;}
          	Pick a rainbow $x$ in $C$ ($x$ is \outr in $G(\Sigma)$)\;
			Run Algorithm \ref{alg:2} on $(G(\Sigma),x)$\;
            \If{ $G(\Sigma)$ has an obstruction $D$ including $x$}{
            \Return $D$; }
           \septtwotwo{ $\Sigma\longleftarrow \Sigma-x$;}

  }
  \Return{$G(\Sigma)$ has no obstructions.}

  \caption{Finding obstructions.}
   \label{alg:3}
\end{algorithm}
\vspace{0.5cm}

%\begin{appendices}

\section{Pseudolinear  drawings of $K_n$} \label{sec:Kn}

In this section we present a simple proof of \feba{a} characterization of pseudolinear drawings of complete graphs (Theorem \ref{thm:Kn}), equivalent to the ones given in \feba{ \cite{AHPSV15} and \cite{AMRS15}}.% One of the equivalences, Corollary \ref{cor:pseudolinearisfconvex},  is shown at the end of this of this section. The other equivalence (Lemma \ref{lemma:fconvex_crossing_sides}) is deferred to Subsection \ref{subsec:f-convex} of Chapter \ref{ch:pseudospherical},  where it  plays an important role in the study of pseudospherical drawings of $K_n$.  

\begin{theorem}\label{thm:Kn}
A good drawing of a complete graph is pseudolinear if and only if it does not include the $B$ configuration (see Figure \ref{B_and_W}).
\end{theorem}
\begin{proof}
\setcounter{claim}{0}

The unique cycle in a $B$ configuration is an obstruction,  so, by Theorem \ref{MAIN},  no pseudolinear drawing of $K_n$ can include it. Conversely, suppose that $D$ is a good drawing of $K_n$ that is not pseudolinear. 
Let $\Sigma=\{D[e]\;:\; e\in E(K_n)\}$ be the set of edge-arcs, and let $G(\Sigma)$ be its underlying plane graph. In order to avoid confusion between vertices and edges of $K_n$ and   $G(\Sigma)$,   vertices in  $G(\Sigma)$ are called {\em points}, and edges of $G(\Sigma)$ are {\em segments}. Because $D$ is good, each point is either in $V(K_n)$ or  a crossing. 

 For every cycle $C$ in $G(\Sigma)$, we let $\delta(C)$  be the set
 \junefif{of}
 points in $C$ for which their two incident segments in $D$ belong to distinct edges in $\Sigma$. Theorem \ref{MAIN} implies that $G(\Sigma)$ has an obstruction $C$. We choose our obstruction  $C$ so that $|\delta(C)|$ is  as small as possible.

%it satisfies the following three conditions: (i) the number of edge-arcs having a segment involved in $C$ is as small as possible; (ii) subject to the previous item, for each edge-arc $D[e]$ involved in $C$, the number segments of $C$ included in $D[e]$ is as small as possible;  and (iii) subject to the previous items, the disk bounded by $C$ is as large as possible.

Since $D$ is good, $|\delta(C)|\geq 3$ and, because $C$ is an obstruction, at most two vertices in $\delta(C)$ are rainbows in $C$. Consider a point $x\in \delta(C)$ \ther{that is} reflecting inside $C$. Note that $x$ is a crossing. Let $\sigma_1$ and $\sigma_2$ be the two edge-arcs in $\Sigma$ crossed at $x$. We traverse $\sigma_1$, starting at $x$, continuing on the segment of $\sigma_1$ included in the interior of $C$, until an end $a_1\in V(K_n)$ of $\sigma_1$ is reached. Likewise we define $a_2$ for $\sigma_2$. Henceforth, we refer to $a_1$ and $a_2$ as the {\em internal vertices} corresponding to the crossing $x$. The following claim explains  why we call them ``internal''. 

\begin{claim}\label{claim:internal}
Let $x\in \delta(C)$ be a point reflecting inside $C$. Then the two internal vertices corresponding to $x$ are in the \junefif{interior} of $C$. 
\end{claim}
\begin{proof}
Let \junefif{$a_1$}
be an internal vertex corresponding to $x$, and suppose 	
\junefif{$\sigma_1$} is the edge-arc including both $x$ and $a_1$. Let $\sigma_1'$ be the substring of $\sigma$, having $x$ and $a_1$ as endpoints. Applying Observation \ref{obs:min_delta} to our obstruction $C$, with $\sigma=\sigma_1$ and $\sigma'=\sigma_1'$, we obtain  that $\sigma_1'\cap C=\{x\}$. Since  points of $\sigma_1'$ near $x$ are in the interior face of $C$,   $\sigma_1'\setminus\{x\}$ is included \junefif{in the interior} of $C$. In particular,  $a_1$ is in such a face. 
\begin{comment}

Suppose that one of the internal vertices corresponding to $x$, say $a_1$, is not in the interior of $C$.

If $\sigma_1$ is the edge-arc including both $x$ and $a_1$, then $\sigma_1$ has a subpath $P$ inside $C$, connecting $x$ to a point $a_1'\in V(C)$, so that $P$ is internally disjoint to $C$. Let $C_1$ and $C_2$ be the cycles obtained from the union of $P$ and one of the two  $xa_1'$-paths in $C$. We may assume that $C_1$ is the one including the two segments of $\sigma_1$ incident to $x$. 

Note that $S(C)$ has points that are interior vertices in each of the two $xa_1'$-paths of $C$, as otherwise, $\sigma_1$ cross another edge twice.
Each of $S(C_1)$, $S(C_2)$ has less vertices than $S(C)$, and thus $C_1$ and $C_2$ are not obstructions.
All the points in $P-a_1'$ are reflecting inside $C_1$, and there are at most two points in  $C_1-P$ not reflecting inside $C_1$ (because such points are also reflecting inside $C$ and $C$ is an obstruction), so $C_1$ has exactly three points  not reflecting inside $C_1$, two of them in $C_1-P$, and the third one is $a_1'$. 

None of the internal vertices of $P$ reflects inside $C_2$. This implies that $C_2$ has a point in $C_2-P$ not reflecting inside $C_2$. However, this point together with the two points  in $C_1-P$ not reflecting inside $C_1$,  are three points not reflecting in $C$, contradicting that $C$ is an obstruction.
\end{comment}
\end{proof}

Now we look at the points in $\delta(C)$ that are not reflecting inside $C$. If $x$ is one of them, then $x$ is
a vertex or a crossing. Suppose that $x$ is a crossing. Let $\sigma_1$, $\sigma_2$ be the edge-arcs crossing at $x$. Because $x$ is not reflecting inside $C$, one of the two segments at $x$ included in $\sigma_1$ is in the outer face of $C$. We traverse $\sigma_1$, starting in $x$, and continuing in the outer face until we reach an end $b_1$ of $\sigma_1$. Likewise we define $b_2$ for $\sigma_2$. These vertices $b_1$, $b_2$ are the {\em external} vertices corresponding to the crossing $x$. 

\begin{claim}\label{claim:external} Let $x$ be a crossing in $\delta(C)$ that is not reflecting inside $C$, and let $\sigma$ be an edge-arc including $x$  and an external vertex $b$ of $x$. If $\sigma'$ is the substring of $\sigma$ connecting $x$ to $b$, then  $\sigma'\setminus \{x\}$ is included in the outer face of $C$. 
%In particular $a$ is in
%Let $\sigma$ be the edge-arc connecting a 
%The external vertices corresponding to a point $x\in S(C)$ that is not reflecting inside $C$ are either in $C$ or in  outer face of $C$. {\bf need to make more explicit this claim, also saying that the -edge arc connecting is internally disjoint from $C$}
\end{claim}
\begin{proof}
Applying Observation \ref{obs:min_delta} to  $C$, $\sigma_1$,  and $\sigma'$, we see that $\sigma'\cap C=\{x\}$. Since the points of $\sigma_1'$ near $x$ are in the outer face of $C$,   $\sigma_1'\setminus\{x\}$ is included in the outer face of $C$. 
%Suppose otherwise, that an external vertex $b_1$ of $x$ is in the interior face of $C$. By following the edge arc $\sigma_1$ connecting $x$ to $b_1$, we can find  $P$ a subpath of $\sigma_1$ that is internally disjoint to $C$, and connects $x$ to point  in $C-x$.  It is not hard to see that outer cycle $C'$ of $C\cup P$ has at most many points reflecting inside as   $C$. Because $\sigma_1$ intersect every edge-arc at most once, $S(C)$ has points in the interior of the two $xb_1'$-paths in $C$. Thus $|S(C')|<|S(C)|$ and $C'$ is an obstruction, contradicting the minimality of $C$.
\end{proof}
\junefif{It is convenient, in  the case  when $x$ is a vertex of $K_n$, to let $x$  be its own external vertex.}

\junefif{Henceforth} we refer to the vertices of $K_n$ that are internal to some crossing in $C$ as the {\em internal vertices} of $C$, and likewise,  the {\em external vertices} of $C$ are the vertices of $K_n$ that are external to some crossing or to a vertex in $C$. 

\begin{claim}

Every segment in $C$ is included in an edge-arc whose ends are either internal or external vertices of $C$.
\end{claim}
\begin{proof}
Any segment $s$ of $C$ is contained in a subpath $P$ of $C$ whose ends are in $\delta(C)$ but is otherwise disjoint from $\delta(C)$.
 This path $P$ is part of an edge-arc $\sigma\in \Sigma$. Let $a\in V(K_n)$ be one of the ends of $\sigma$, and suppose that $x$ is the first end of $P$ that we encounter when we traverse $\sigma$ from $a$ to the other end of $\sigma$.
 If $\sigma$ is reflecting at $x$, then $a$ is internal. If $\sigma$ is not reflecting at $x$, then $a$ is external.  
%  If $x$ is reflecting inside $C$, then $a$ is an internal vertex of $x$.
  % Alternatively, if $x$ is not reflecting inside $C$, then $a$ is an external vertex of $C$.
    Likewise, the other end of $\sigma$ is internal or external.
\end{proof}

Suppose that  $K_n$ has a vertex $y$ that is \junefif{neither} external nor internal to $C$. Then, by our previous claim,  the underlying plane graph of $D[K_n-y]$ contains a cycle   whose drawing is $D[C]$ and  is an obstruction. Thus,   $D[K_n-y]$ is not pseudolinear,  
and   applying induction on $n$, we obtain that $D[K_n-y]$ has a $B$ configuration.  Henceforth we assume that  all the vertices of $K_n$ are either internal or external to $C$.

\begin{claim}\label{claim:out_cycle}
Either the outer face of $D$ is bounded by a cycle of $K_n$ or $D$ has a $B$ configuration.
\end{claim}
\begin{proof}
Suppose that the outer face of $D$ is not bounded by a cycle of $K_n$. Then the outer face is incident to a crossing $\times$ between two edge-arcs $\sigma_1$ and $\sigma_2$. Let $K$ be the crossing $K_4$ induced by the ends of $\sigma_1$ and $\sigma_2$. The drawing $D[K]$ has exactly five faces, four of them incident to $\times$. Exactly one of the faces incident to $\times$ includes the outer face of $D$. Such a face of $D[K]$ is bounded by portions of $\sigma_1$, $\sigma_2$, and an edge $e$ of $K_n$ connecting an end of $\sigma_1$ to an end of $\sigma_2$. The drawing induced by $\sigma_1$, $\sigma_2$ and $D[e]$ is a $B$ configuration.
\end{proof}

Claims \ref{claim:internal} and \ref{claim:out_cycle} imply that the outer cycle of $D$ consists of only external vertices of $C$. Every external vertex either is associated  with  a crossing  that is not reflecting inside $C$, or  is 
\junefif{itself}
a vertex of $K_n$ in $C$. Because $C$ has at most two points not reflecting inside $C$, and each of them has at most two external vertices,  there are at most four points in the outer cycle of $D$. Thus the outercycle is a 3- or 4-cycle of $K_n$. 

As $C$ has at least three external vertices (in the outer cycle),  $\delta(C)$ has precisely two  points $p$ and $q$ not reflecting inside $C$. The outer cycle of $D$ has an edge $uv$, where $u$ is  external to $p$ and $v$ is external to $q$ (possibly $u=p$ or $q=v$). 

%Every point in $\delta(C)$ not reflecting inside $C$ has associated at most two external vertices, so the outer cycle has an edge joining two vertices $u$ and $v$ associated to distinct non-reflecting points $p, q$  in $S(C)$,  respectively. Possibly $u=p$ or $q=v$. 

Consider the $pq$-path $P$ in $G(\Sigma)$, starting at $p$, continuing on the edge-arc connecting $p$ to $u$, then following the edge $uv$ until we reach $v$, and ending by following the edge-arc connecting $v$ to $q$. We finish our proof by considering two cases, depending on whether $uv$ is a segment of $C$. 
\begin{case}
$uv$ is not a segment of $C$. 
\end{case}

In this case, there exists \marchtwo{a point} $w\in D[uv]\setminus D[C]$. As $D[uv]$ is part of the outer cycle, it   contains \junefif{neither} crossings nor vertices in its interior, so the arcs in $D[uv]$ connecting  $w$ to the ends $u$ and $v$ are internally disjoint from $C$. From Claim \ref{claim:external}, it follows that
the $pu$- and the $qv$-subpaths  of $P$ are internally disjoint from $C$. Thus $P$ is an arc connecting $p$ and $q$ in the outer face of $C$. 

%Since $uv$ is not crossed  it is part of the outercycle, either $D[uv]\cap D[C]\subseteq\{u,v\}$ or $D[uv]\subseteq D[C]$. 

%Suppose first that $D[uv]\cap D[C]\subseteq\{u,v\}$. Observe that, if for instance, $u$ is in $C$, then $u=p$, and analogously $v=q$ when $v\in V(C)$. In any case, $P$ is internally disjoint to $C$, and because $P$ contains the edge $uv$ in the boundary, $D[P]$ is an arc connecting $u$ and $v$ in the outer face of $D[C]$. 

Consider the cycle $C'$ obtained from the union of $P$ and the $pq$-path of $C$ that lies in the outer face of $D[C\cup P]$. 

We will show that $C'$ is an obstruction by showing that $u$ and $v$ are the only rainbows of  $C'$. 
If $p\neq u$, then the edge-arc $\sigma$ connecting $p$ and $u$ shows that every point in $(P-u)\cap\sigma$ is reflecting inside $C'$. Analogously, if $q\neq v$, the points
\junefif{distinct}
from $v$ in the edge-arc connecting $q$ and $v$, are reflecting inside $C'$.  
Thus the internal points in $P$, with the  exception of $u$ and $v$,  are  not reflecting.  
The same holds for the points in $C'-P$, as these points are not reflecting inside $C$ (recall that $p$ and $q$ are the only rainbows of $C$).  Thus $u$ and $v$ are the only rainbows  of $C'$. 

Note that all the segments  of $C'$ are included in edges whose ends are $u$, $v$ or interior points of $C$. So if $y$ is a vertex in the outercycle of $D$ distinct from $u$ and $v$,  $D[K_n -y]$ also includes  $D[C']$ as an obstruction, implying that $D[K_n-y]$ is not pseudolinear. Again, by induction on $n$, we obtain that $K_n-y$ has a $B$ configuration.

\begin{case}
$uv$ is a segment of $C$. 
\end{case}
In this case, as $u$, $v$ are vertices of $K_n$ in $C$, they are rainbows of $C$. Since $p$ and $q$ are the only rainbows,  $p=u$, $q=v$, and $D[uv]$ is a segment of $C$.  Then, all the segments  of \septtwotwo{$C$} are included in \junefif{edge-arcs} whose ends are $u$, $v$ or interior points of $C$. Again, remove a vertex in the outer cycle of $D$ distinct from $u$ and $v$ to obtain a non-pseudolinear drawing of $K_{n-1}$ in which, by induction, we find a $B$ configuration. 
\end{proof}

\section{Concluding remarks}\label{sec:conclusions}
In  our initial attempts to formulate Theorem \ref{MAIN}, we intended to characterize non-pseudolinear good drawings of graphs by means of having at least one of the configurations in Figure \ref{obstructions} as a subdrawing. We obtain this as an easy consequence of Theorem \ref{MAIN}. We sketch its proof.

\minconf*

\begin{proof}
Take $C$ an obstruction of the underlying plane graph associated to $D$. We choose $C$ so that \septtwotwo{$|\delta(C)|$} is as small as possible. Decompose $C$ into a cyclic sequence of paths $P_0,\ldots, P_m$, where $P_i$ connects two points in $\delta(C)$ and it is otherwise disjoint from $\delta(C)$. By using Observation \ref{obs:min_delta}, one can show that $P_0,\ldots, P_m$ belong to distinct edge-arcs $\sigma_0,\ldots,\sigma_m$, respectively. For each $P_i$, we consider the string $\sigma_i'$, obtained by slightly extending the ends of $P_i$ that are reflecting in $C$;  we extend  them along  $\sigma_i$. 

Let $x\in \delta(C)$ be an end shared by $P_{i-1}$ and $P_i$. If $x$ is  reflecting in $C$, then $x$ is a crossing between $\sigma_{i-1}$ and $\sigma_i$. Moreover, the arcs added to $P_{i-1}$ and $P_i$  at $x$ to obtain $\sigma_{i-1}'$ and $\sigma_i'$ are in the interior of $C$. If $x$ is a rainbow in $C$, then $P_i$ and $P_{i-1}$ are not extended at  $x$, and $x$ acts as one of the black dots in Figure \ref{obstructions}. The rest of the points in $\delta(C)$ are crossings in  $\bigcup_{i=0}^m\sigma_{i}'$ facing the interior of $C$. 
Since $C$ has at most two rainbows,    $\bigcup_{i=0}^m\sigma_{i}'$ is one in Figure \ref{obstructions}. 
\end{proof}

There are 
\junefif{pseudolinear}
drawings that are not stretchable. For instance, consider the Non-Pappus configuration in Figure \ref{nonpappus}. Nevertheless, as an immediate consequence of Thomassen's  main result in \cite{T88},  pseudolinear and stretchable drawings are equivalent, under the assumption that every edge is crossed at most once. 

\begin{figure}[ht]

    \centering
    
    \includegraphics[scale=0.4]{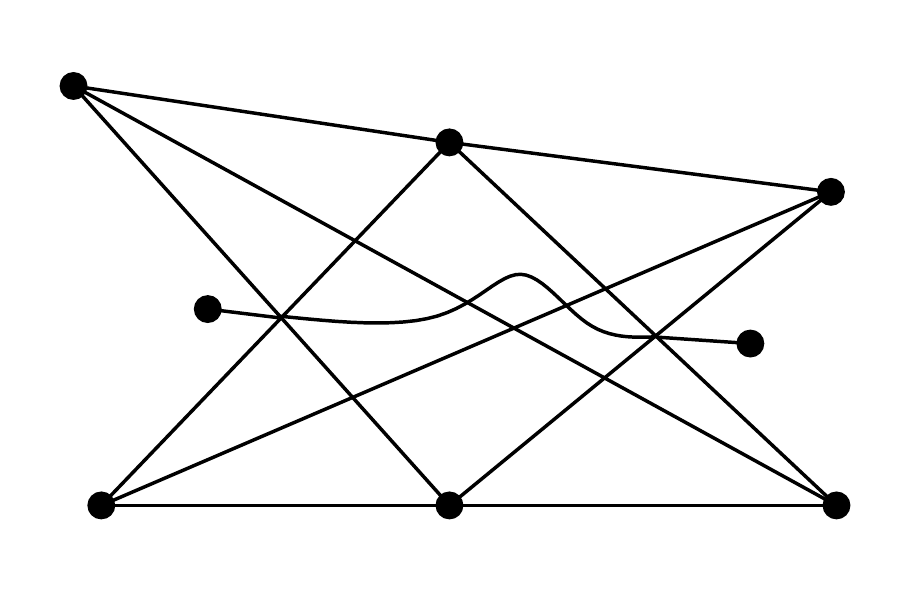}
    \caption{Non-Pappus configuration.}
 	\label{nonpappus}
\end{figure}

\begin{corollary}
A drawing of a graph in which every edge is crossed at most once is stretchable if and only it is pseudolinear. 
\end{corollary}
\begin{proof}
Let $D$ be a drawing of a graph in which every edge is crossed at most once. If $D$ is stretchable then clearly it is pseudolinear. To show the converse, suppose that $D$ is pseudolinear.  Then $D$ does not contain any obstruction, and in particular, \junefif{neither of the} $B$ and $W$ configurations in Figure \ref{B_and_W} occur in $D$. In \cite{T88}, it was shown that not containing the  $B$ and $W$ configurations is equivalent to  being rectilinear.  
\end{proof}

One can construct more general  examples of pseudolinear drawings that are not stretchable
by considering  non-strechable arrangements of pseudolines.  However,  such examples seem to inevitably have edges crossing several times. This leads to two natural questions.

\begin{question}\label{Q1}
Is it true that if $D$ is a pseudolinear drawing in which every edge is crossed at most twice, then $D$ is stretchable?
\end{question}
\begin{question}\label{Q2}
Is it true that if $D$ is a pseudolinear drawing in which all the crossings involve a fixed edge, then $D$ is stretchable?
\end{question}

%%%%%paste

\medskip
 
\bibliographystyle{abbrv}
\bibliography{Mendeley}

\end{document}